\documentclass[11pt,a4paper]{amsart}
\usepackage{geometry}
\usepackage{graphicx}
\usepackage{amsmath,amssymb}
\usepackage{amsthm}
	\theoremstyle{definition}
		\newtheorem{dfn}{Definition}[section]

		\theoremstyle{plain}
		\newtheorem{thm}[dfn]{Theorem}
		\newtheorem{lem}[dfn]{Lemma}
		\newtheorem{prop}[dfn]{Proposition}

\makeatletter
	
	\@addtoreset{equation}{section}
\newcommand{\rep}{{\rm rep}}
\newcommand{\tr}{{\mathrm{tr}}}
\newcommand{\ol}{\overline}
\allowdisplaybreaks[3]

\title[A gap of the exponents of repetitions of Sturmian words]
	{A gap of the exponents of repetitions \\ of Sturmian words}

\author[S. Ohnaka]{Suzue Ohnaka}

\author[T. Watanabe]{Takao Watanabe}
\thanks{The second author was supported in part by JSPS KAKENHI Grant Number 17K05170}

\address[S. Ohnaka]
{Naruo Senior High School,  Gakubundencho 2-1-60
 Nishinomiya, Hyougo, 663-8182, Japan}
 \email{u330722e@alumni.osaka-u.ac.jp}
\address[T. Watanabe]
{Department of Mathematics, Graduate School of Science, Osaka University, 
	Toyonaka, Osaka, 560-0043, Japan}
\email{twatanabe@math.sci.osaka-u.ac.jp}

\keywords{Combinatorics on words, Sturmian word, irrationality exponent}
\subjclass[2010]{Primary 68R15; Secondary 11A55, 11A63}

\begin{document}

\maketitle

\begin{abstract}
By measuring second occurring times of factors of an infinite word $x$, 
Bugeaud and Kim introduced a new quantity $\rep(x)$ called the exponent of repetition of $x$. 
It was proved by Bugeaud and Kim that $1 \leq \rep(x) \leq r_{\max} = \sqrt{10} - 3/2$ if  $x$ is a Sturmian words.  
In this paper, we determine the value $r_1$ such that there is no Sturmian word  $x$ satisfying 
$r_1 < \rep(x) < r_{\max}$ and $r_1$ is an accumulation point of the set of $\rep(x)$ 
when $x$ runs over the Sturmian words. 
\end{abstract}

\section*{Introduction}

A finite or infinite sequence $x =x_1x_2 \cdots$ of elements of a finite set $\mathbf{A}$ is called a word over $\mathbf{A}$. A subword of consecutive letters occurring in $x$ is called a factor of $x$. For an infinite word $x$ and a positive integer $n$,  we denote by $p(n,x)$ the number of different factors  of length $n$ occurring in $x$. It is known by Morse and Hedlund \cite[\S 7]{morse} or Coven and Hedlund \cite[Theorem 2.06]{coven} that the subword complexity function $n \mapsto p(n,x)$  characterizes  the eventually periodic words. Namely, the following statements for an infinite word $x$ are equivalent:
\begin{itemize}
\item $x$ is  eventually periodic;
\item $p(n,x) \le n$ for some $n$;
\item $\{p(n,x)\}_{n \ge 1}$ is bounded;
\end{itemize} 
(see, e.g., \cite[Theorem 10.2.6]{allouche}).  
This implies that any non-eventually periodic infinite word $x$ satisfies $p(n,x) \ge n+1$ for all positive integers $n$. An infinite word $x$ satisfying $p(n,x) = n+1$ for all $n$ is called a Sturmian word. Sturmian words appear in many different areas and have been studied from different viewpoints. (See, e.g., \cite[Notes on Chapter 10]{allouche}, \cite[Notes on Chapter 2]{lothaire}).  
 
Recently, Bugeaud and Kim introduced a new complexity function $r(n,x)$  in \cite{bugeaud}. For  an infinite word $x$, we denote the factor $x_i x_{i+1} \cdots x_{i+j}$  of $x$ by  $x_{i}^{i+j}$. Then $r(n,x)$  is defined by 
 \begin{gather*}
	r(n, x) = \min\{ m \geq 1 \; \mid \; x_{i}^{i+n-1} = x_{m-n+1}^{m} \;\; 
		\text{for some $i$ with $1 \leq i \leq m-n$}\}.
\end{gather*}
In other words, $r(n, x)$ denotes the length of the smallest prefix of $x$ 
containing two occurrences of some factor of length $n$. 
For example, if $x$  is a binary infinite word over $\{0,1\}$ with a prefix $x_1 \cdots x_8 = 01001001$, then we have $r(2,x) = 5$ since the factor $01$ occurs twice in the prefix of length $5$. Similarly, we have $r(3,x) = 6$ by the factor $010$ and $r(4,x) = 7$ by the factor $0100$. We note that the factor $0100$ occurs twice as $x_1x_2x_3x_4$ and $x_4x_5x_6x_7$ in the prefix  of length $7$ by admitting an overlap of $x_4$. The following properties of $r(n,x)$ were proved by Bugeaud and Kim.

\begin{thm}[{\cite[Lemma 2.1, Theorems 2.3 and 2.4]{bugeaud}}] \label{bugthm:2.3}
Let $x$ be an infinite word. 

$(1)$  $n + 1 \le r(n,x) \le p(n,x) + n$ and $r(n,x) \le r(n+1,x) - 1$ hold for all $n \ge 1$. 

$(2)$  $x$ is eventually periodic if  and only if $r(n,x) \le 2n$ holds for all sufficiently large $n$. 

$(3)$  $x$ is a Sturmian word if and only if $x$ satisfies $r(n,x) \le 2n+1$ for  $n \ge 1$ and  $r(n,x) = 2n+1$ holds for infinitely many $n$.    
\end{thm}
 
 By Theorem \ref{bugthm:2.3} (1), every infinite word $x$ satisfies 
\begin{gather*}
 1 + \frac{1}{n} \le \frac{r(n,x)}{n} \le 1 + \frac{p(n,x)}{n}  
\end{gather*}
 for all $n \ge 1$. If $x$ is eventually periodic, then $p(n,x)/n$ converges to $0$ as $n \to \infty$ since $\{p(n,x)\}_{n \ge 1}$ is bounded. Therefore, $r(n,x)/n$ converges to $1$ as $n \to \infty$ in this case.  In general, $r(n,x)/n$ does not necessarily converge as $n \to \infty$ (see Figure 1 of \S 1).  Bugeaud and Kim defined the exponent $\mathrm{rep}(x)$ of repetitions  of $x$ as 
 \begin{gather*}
 \mathrm{rep}(x) = \liminf_{n \to \infty} \frac{r(n,x)}{n}  .
 \end{gather*}
 A significant meaning of $\rep(x)$ is the following connection to Diophantine approximations.  For an irrational real number $\xi$, the irrationality exponent $\mu(\xi)$ of $\xi$ is defined to be the supremum of the real numbers $\mu$ such that the inequality
\begin{gather*}
\left | \xi - \frac{p}{q} \right | < \frac{1}{q^\mu} 
\end{gather*} 
 has infinitely many solutions in rational numbers $p/q$.  Famous Roth's theorem says that $\xi$ is transcendental if $\mu(\xi) > 2$. An infinite word $x = x_1x_2 \cdots $ over $\{0, 1, \cdots, b-1\}$ for an integer $b \ge 2$ defines the real number 
\begin{gather*}
 \xi_{x,b} = \sum_{k=1}^\infty \frac{x_k}{b^k}, 
\end{gather*}
 which is irrational if $x$ is not eventually periodic. Then Bugeaud and Kim proved the following:
 
\begin{thm}[{\cite[Theorems 4.2 and 4.5]{bugeaud}}] \label{bugthm:4.2}
Let $x$ be an infinite word over $\{0,1, \cdots, b-1\}$ which is not eventually periodic. Then $\xi_{x,b}$ satisfies 
\begin{gather}
\mu(\xi_{x,b}) \ge \dfrac{\mathrm{rep}(x)}{\mathrm{rep}(x) - 1} \,, \label{eq:irrationality}
\end{gather}
where the right hand side is infinite if $\mathrm{rep}(x) = 1$. Moreover, if $x$ is a Sturmian word, then the equality of (\ref{eq:irrationality}) holds. 
\end{thm}

See \cite[\S 4]{bugeaud} for more information about previous researches on relations between the combinatorial properties of words and the $b$-ary expansion of an irrational number.
   
It will be natural to ask  what possible values of $\mathrm{rep}(x)$ are. We denote by $\textbf{St}$ the set of all Sturmian words over $\mathbf{A}$ and by $\mathrm{rep}(\mathbf{St})$ the subset  $\{ \mathrm{rep}(x) \;| \; x \in \mathbf{St} \}$ of the real numbers.  In the same paper \cite{bugeaud}, both the maximum and the minimum values of $\mathrm{rep}(\mathbf{St})$ were determined.

\begin{thm}[{\cite[Theorems 3.3 and 3.4]{bugeaud}}] \label{bugthm:3.3}
Every Sturmian word $x$ satisfies 
\begin{gather*}
	1 \leq \rep(x) \leq r_{\max} := \sqrt{10} - \frac{3}{2} = 1.66227\cdots .
\end{gather*}
$r_{\max}$ is the maximum and $1$ is the minimum of $\mathrm{rep}(\mathbf{St})$. If $x$ satisfies $\rep(x) = r_{\max}$, then the continued fraction expansion of the slope of $x$ is 
of the form $[0,a_1, a_2, \cdots, a_K, \overline{2,1,1}]$ for some $K$.
\end{thm}
 
As noticed in  \cite[below Proof of Theorem 3.4]{bugeaud}, 
$r_{\max}$ is an isolated point of $\rep(\mathbf{St})$. On the other hand, the minimum $1$ is an accumulation point of $\rep(\mathbf{St})$ (\cite[Proof of Theorem 3.3]{bugeaud}). Except for these results, little is known about  the set $\mathrm{rep}(\mathbf{St})$. In this paper, we find a gap in  $\rep(\mathbf{St})$ and determine the largest accumulation point. Main results of this paper are the following two theorems.  

\begin{thm}
	Let 
	\begin{gather*}
		r_1 := \frac{48 + \sqrt{10}}{31} = 1.65039 \cdots . 
	\end{gather*}
	Then there is no Sturmian word $x$ satisfying
	\begin{gather*}
		r_1 < \rep(x) < r_{\max} . 
	\end{gather*}
	The value $r_1$ is an accumulation point of  $\rep(\mathbf{St})$. 
	\end{thm}
 
 \begin{thm}
	Let 
	\begin{gather*}
		r_2 := \frac{415\sqrt{149} - 2693}{1438} = 1.65001\cdots .
	\end{gather*}
	If $x \in \mathbf{St}$ satisfies $r_2 \leq \rep(x) \leq r_1$, 
	then the continued fraction expansion of the slope of $x$ is of the form 
	\begin{gather*}
		[0,a_1,a_2, \cdots, a_K, (2,1,1)^{n_1}, 1, (2,1,1)^{n_2}, 1, \cdots ]
	\end{gather*}
	for some $K$ and some sequence $n_1, n_2, \cdots $ of integers greater than $1$, 
	where $(2,1,1)^{n_i}$ denotes the periodic sequence repeating $2,1,1$ $n_i$ times . 
\end{thm}
 
We do not know whether $r_1$ is contained in $\mathrm{rep}(\mathbf{St})$ or not.  These results  may remind us Lagrange (or Markov) spectrums, but it is not clear  if  there are any similarities  between $\mathrm{rep}(\mathbf{St})$ and Lagrange spectrums. 
 
It is possible to compute $\mathrm{rep}(x)$ when the slope of $x$ is a quadratic irrational whose continued fraction expansion has  a short period and $x$ satisfies some additional condition on prefixes. The value $r_2$ is attained on $\rep(x)$ of some $x \in \mathbf{St}$ with the slope $[0,\overline{2,1,1,2,1,1,1}]$ (Lemma \ref{lem:r2}). In contrast, the value $r_1$ is the limit of $\mathrm{rep}(x^{(n)})$ as $n \to \infty$ of some sequence $\{x^{(n)}\}_{n \ge 1}$ in $\mathbf{St}$.  The slope of $x^{(n)}$ is of the form 
\begin{gather*}
[0,a_1, \cdots, a_K, \overline{(2,1,1)^n, 2,1,1,1}].
\end{gather*} 
We will devote \S \ref{sec:accumulate} to compute $\lim_{n \to \infty}\mathrm{rep}(x^{(n)})$ after a preparatory section \S \ref{sec:continuedfraction} on continued fractions. Once one knows the value $r_1$, it is not difficult to show that the open interval $(r_1, r_{\max})$ is  a gap in $\mathrm{rep}(\mathbf{St})$. By this reason, the first assertion of Theorem 0.4 and Theorem 0.5 will be proved in \S \ref{sec:r2}. The second assertion of Theorem 0.4 will be proved in \S \ref{sec:largest}.

 \vskip 10mm 
\section{Sturmian words}
Let $x=x_1x_2\cdots $ be a word over a finite set $\mathbf{A}$. A subword of consecutive letters in $x$ is called a factor of $x$. If both $y$ and $z$ are factors of $x$ and $x$ is a concatenation of $y$ and $z$, i.e., $x = yz$, then $y$ is called a prefix of $x$ and $z$ is called a suffix of $x$.  If $x = x_1 \cdots x_\ell$ is a finite word, then the number $\ell$ is called the {\it length} of $x$ and is denoted by $|x|$.  

Since every Sturmian word is a binary word, it is sufficient to consider only Sturmian words over the alphabet $\mathbf{A} = \{0,1\}$ in order to study $\mathrm{rep}(\mathbf{St})$.  Then every element of $\mathbf{St}$ is described as follows. For an irrational real number $\theta \in (0, 1)$ and a real number $\rho$, we set 
\begin{gather*}
	s_{\theta, \rho}(n)  = \lfloor \theta(n+1)+\rho \rfloor-\lfloor \theta n+\rho \rfloor  \quad \text{and} \quad S_{\theta, \rho}(n)  = \lceil \theta(n+1)+\rho \rceil-\lceil \theta n+\rho \rceil
\end{gather*}
for $n \ge 1$. Then both $s_{\theta, \rho} :=  s_{\theta, \rho}(1)s_{\theta,\rho}(2) \cdots $ and  $S_{\theta, \rho} :=  S_{\theta, \rho}(1)S_{\theta,\rho}(2) \cdots $ are Sturmian words over $\mathbf{A}$. Conversely, for any $x \in \mathbf{St}$, there exists an irrational $\theta \in (0,1)$ and  a real number $\rho$ such that  $x = s_{\theta, \rho}$ or $x = S_{\theta, \rho}$. This $\theta$ is called the slope of $x$. The slope $\theta$ of $x$ is uniquely determined by 
\begin{gather*}
 \theta = \lim_{n \to \infty} \dfrac{(\text{the number of digit $1$ in $x_1x_2 \cdots x_n$})}{n}. 
\end{gather*} 
For some fundamental properties of Sturmian words, e.g., see  \cite[\S 8]{aigner}, \cite[Chapter 10]{allouche}, \cite{arnoux}, \cite[Chapter 2]{lothaire} . The continued fraction expansion of the slope of a Sturmian word  plays an important role. Any irrational real number  $\theta$ has a unique continued fraction expression such as 
\begin{gather*}
	\theta =a_0+\dfrac{1}{a_1+\dfrac{1}{a_2+\dfrac{1}{\ddots}}}  = [a_0,a_1,a_2,\cdots]\,, 
\end{gather*}
where $a_0$ is an integer and $a_1, a_2, \cdots $ are positive integers. For each $n \geq 1$, the fraction 
\begin{gather*}
	\dfrac{p_n}{q_n} = [a_0, a_1, \cdots, a_n]
\end{gather*}
is called the  $n$-th convergent of $\theta$. As usual, we set $(p_0,q_0) =(a_0,1)$ and $(p_{-1}, q_{-1}) = (1,0)$. Then $(p_n, q_n)$ satisfies the recurrence relation
\begin{gather*}
p_n   = a_np_{n-1} + p_{n-2}, \qquad 
q_n   = a_nq_{n-1} + q_{n-2} 
\end{gather*}
for $n \ge 1$. 

By Theorem 0.1 (1) and (3), every Sturmian word $x$ satisfies 
\begin{gather*}
 1 + \frac{1}{n} \le \frac{r(n,x)}{n} \le 2 + \frac{1}{n} 
\end{gather*}
for all $n \ge 1$, and $r(n, x) /n= 2 + 1/n$ holds for infinitely many $n$. Therefore, 
\begin{gather*}
\limsup_{n \to \infty} \frac{r(n,x)}{n} = 2
\end{gather*}
holds for all $x \in \mathbf{St}$. If $r(n,x) \ne 2n+1$ for some $n$ , then $r(n,x) \le 2n$ since $r(n,x)$ is a positive integer.  In this case, by \cite[Lemma 3.5]{bugeaud} and  Theorem 0.1(1), we have 
\begin{gather*}
 r(n,x) = r(n-1, x) + 1 ,
\end{gather*}
and then
$$ \frac{r(n,x)}{n} = \frac{r(n-1,x)}{n-1} + \frac{1}{n} \left ( 1 - \frac{r(n-1,x)}{n-1} \right ) < \frac{r(n-1,x)}{n-1} \,. $$
Namely, the function $n \mapsto r(n,x)/n$ is a piecewise decreasing function. Figure 1 illustrates a behavior of the function $n \to r(n,x)/n$ when $x= s_{\alpha, 1/3}$ with $\alpha= [0, \overline{2,1,1}]$. It is known by \cite[\S 7]{bugeaud} that $\mathrm{rep}(s_{\alpha, 1/3}) = r_{\max}$. It is not yet known how the intercept $\rho$  of a Sturmian word $x = s_{\theta,\rho}$ or $x = S_{\theta, \rho}$  affects $\mathrm{rep}(x)$. 

\begin{figure}[!htb]
\centering
\includegraphics[width=10cm,clip]{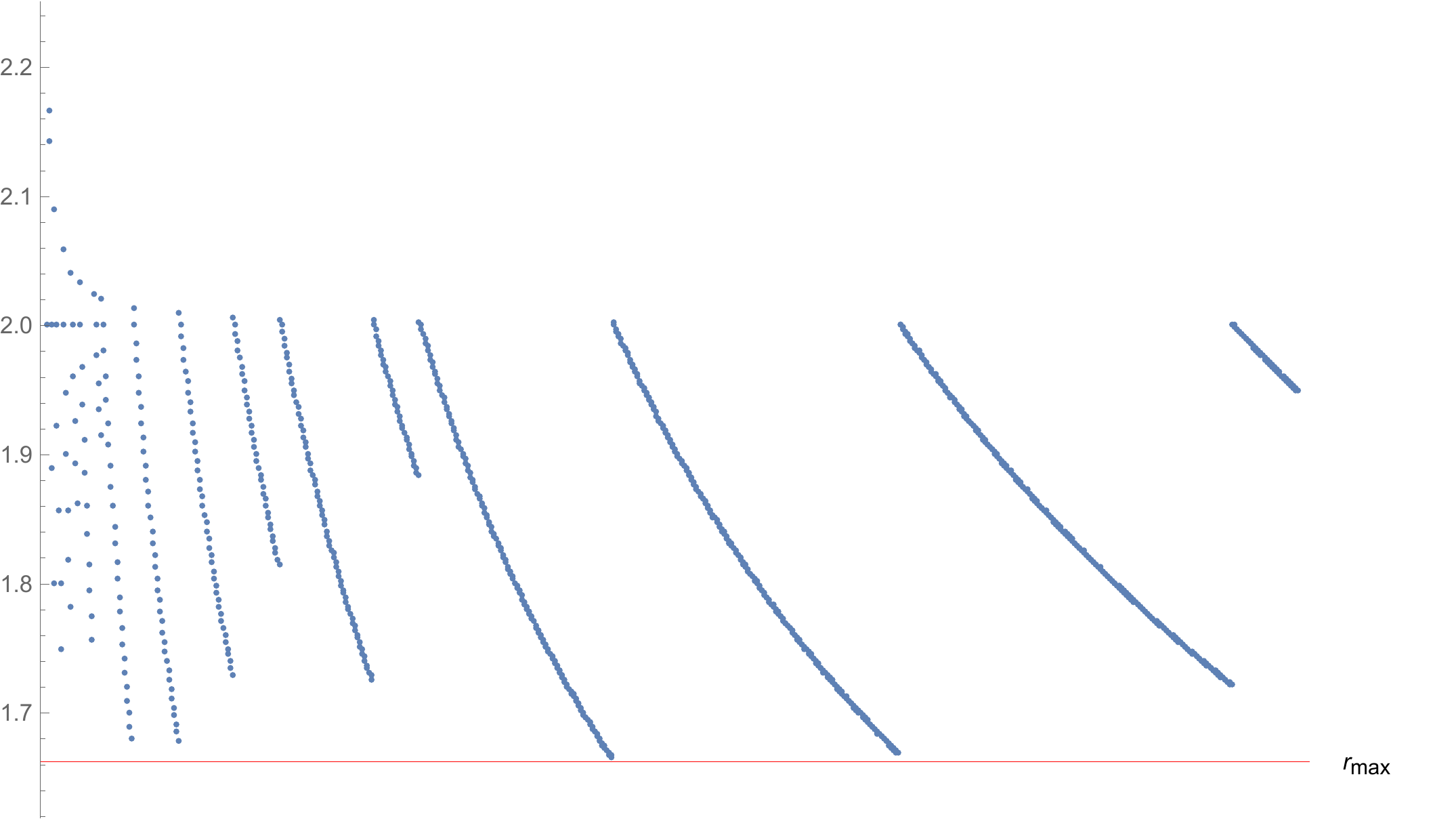}
\caption{Plot of $r(s_{\alpha,1/3}, n)/n$ for $1 \le n \le 1000$}
\end{figure}

For an irrational real number $\theta = [0,a_1,a_2, \cdots ]$, the Sturmian word $s_{\theta, 0}$ of  intercept $0$ is given by the limit of some finite words. For a sequence $\{ a_k \}_{k \geq 1}$ of positive integers, we define inductively a sequence $\{ M_k \}_{k \geq 0}$ of finite words  over $\mathbf{A}$ as follows:
\begin{gather*}
	M_0=0, \qquad M_1=0^{a_1-1}1, \qquad M_{k+1}=M^{a_{k+1}}_k M_{k-1} \;\;(k \geq 1) . 
\end{gather*}
This $\{ M_k \}_{k \ge 0}$ is called the characteristic block defined from $\{a_k\}_{k \ge 1}$. 
Let $\{ p_k/q_k \}_{k \ge 0}$ be the sequence of convergents of $\theta$. Then $q_k = |M_k|$ and $p_k$ equals the number of digit $1$ in $M_k$ for all $k \geq 0$. Every $M_k$ is a prefix of $s_{\theta, 0}$ (see, e.g., \cite[Theorem 8.33]{aigner}),   namely we have
\begin{gather*}
	s_{\theta, 0} = \displaystyle \lim_{k \to \infty} M_k . 
\end{gather*}
The Sturmian word $s_{\theta, 0}$ is called the  characteristic word of slope $\theta$.
It is known that only the last two letters of $M_{k+1} M_k$ and $M_k M_{k+1}$ are different (see e.g.,  \cite[Proposition 2.2.2]{lothaire}). 
For a non-empty finite word $U$, 
we write $U^{-}$ for the word removed the last letter of $U$. 
For each $k \geq 1$, we set 
\begin{gather*}
	\widetilde{M}_k=(M_k M_{k-1})^{--}=(M_{k-1}M_k)^{--}
\end{gather*}
and observe that $\widetilde{M}_k$ is a prefix of $M_{k+1}$. 

The following lemma proved in \cite[\S7]{bugeaud} plays a key role. 

\begin{lem}[{\cite[Lemma 7.2]{bugeaud}}] \label{buglem:7.2}
	Let $x$ be a Sturmian word of slope $\theta = [0,a_1,a_2, \cdots]$, $\{ p_k/q_k \}_{k \ge 0}$ be the sequence of convergents of $\theta$ and $\{M_k\}_{k \ge 0}$ be the characteristic block defined from $\{a_k\}_{k \ge 1}$. 
	Then, for each $k \geq 1$, there exists a unique finite word $W_k$ satisfying 
	\begin{enumerate}
		\item[$\text{[1]}_k$] $x=W_k M_k \widetilde{M}_k \cdots$,\; where $W_k$ is a non-empty suffix of $M_k$, \\
		or
		\item[$\text{[2]}_k$] $x=W_k M_{k-1} M_k \widetilde{M}_k \cdots$,\; where $W_k$ is a non-empty suffix of $M_k$, \\
		or
		\item[$\text{[3]}_k$] $x=W_k M_k \widetilde{M}_k \cdots$,\; where $W_k$ is a non-empty suffix of $M_{k-1}$, 
	\end{enumerate}
	and all the $(2 q_k + q_{k-1})$ cases are mutually exclusive. 
	Furthermore, we have $W_{k+1}=W_k M_{k-1}$ if the case $[2]_k$ holds and $W_{k+1}=W_k$ if the case $[3]_k$ holds . 
\end{lem}

We continue the notation in Lemma \ref{buglem:7.2}. For a given Sturmian word $x$, we use the following notation throughout this paper:  
\begin{gather*}
	\eta_k:=\dfrac{q_{k-1}}{q_k}, \qquad 
	t_k:=\dfrac{|W_k|}{q_k}, \qquad
	\varepsilon_k:=\dfrac{2}{q_k}.	
\end{gather*}
Here the dependency of $x$ is omitted from these notations to simplify descriptions. 
The Golden Ratio $(1+\sqrt{5})/2 = 1.6180 \cdots$ is denoted by $\varphi$.  
 We will frequently use the following lemmas. 

\begin{lem}[{\cite[Lemma 7.4]{bugeaud}}] \label{buglem:7.4}
	Let $x$ be a Sturmian word. 
	\begin{enumerate}
		\item If the case $[1]_k$ holds, then 
		\begin{gather*}
		\dfrac{r(n, x)}{n}<\varphi +2 \varepsilon_k
		\end{gather*}	
			for some $n$ with $q_k-2 \leq n \leq |W_k|+q_k+q_{k-1}-2$. 
		\item If the case $[3]_k$ holds,  then 
			\begin{gather*}
			\dfrac{r(n, x)}{n}<\varphi +2 \varepsilon_k
			\end{gather*}
			for some $n$ with $|W_k|+q_k-2 \leq n \leq |W_k|+q_{k+1}+q_k-2$. 
	\end{enumerate}
\end{lem}

\begin{lem}[{\cite[Lemma 7.5]{bugeaud}}] \label{buglem:7.5}
	Let $x$ be a Sturmian word of slope $\theta = [0,a_1,a_2, \cdots]$. Assume that the case $[2]_k$ holds  and $a_k \geq 3$ for some $k$. If $k$ is sufficiently large, then 
		\begin{gather*}
		\dfrac{r(n, x)}{n}<\dfrac{\sqrt{17}+9}{8}+2\varepsilon_k
			=1.6403 \cdots +2\varepsilon_k
	\end{gather*}
	for some integer $n$ with $q_k/2-2 \leq n \leq q_k+q_{k-1}-2$.  
\end{lem}

\begin{lem}[{\cite[Lemma 7.6]{bugeaud}}] \label{buglem:7.6}
	Let $x$ be a Sturmian word. If the case $[2]_k$ holds for some sufficiently large $k$, then we have 
	\begin{enumerate}
		\item $\dfrac{r(|W_k|+q_k+q_{k-1}-2, x)}{|W_k|+q_k+q_{k-1}-2}
			<1+\dfrac{1+\eta_k}{t_k+1+\eta_k}+\varepsilon_k$, 
		\item $\dfrac{r(q_k+q_{k-1}-2, x)}{q_k+q_{k-1}-2}
			<1+\dfrac{t_k+\eta_k}{1+\eta_k}+\varepsilon_k$. 
	\end{enumerate}
\end{lem}

\vskip 10mm
\section{A gap in $\mathrm{rep}(\mathbf{St})$} \label{sec:r2}

 First, we show that some lower bound of $\mathrm{rep}(x)$ gives a restriction on  partial quotients of the continued fraction expansion of the slope of $x$. In this section, for convenience sake, we often regard a finite sequence of one digit numbers as a finite word over $\{0,1, \cdots, 9\}$.  A sequence $a_1, \cdots, a_\ell$ of one digit numbers is denoted by $a_1 \cdots a_\ell$. Then, the equality $a_1 \cdots a_\ell= b_1 \cdots b_\ell$ of two finite sequences of one digit numbers is the same meaning as $a_i = b_i$ for $i = 1, \cdots, \ell$.  We put 
\begin{gather*}
	r_3 :=\dfrac{2(1869+2\varphi)}{2277}=1.64448\cdots . 
\end{gather*}

\begin{lem} \label{lem:21111}
	Let $x$ be a Sturmian word and $\theta = [0, a_1, a_2, \cdots]$ be the slope of $x$. 
	If $\rep(x)>r_3$ and $k$ is  sufficiently large, 
	then $a_k \in \{1,2 \}$ and 
	the sequence $a_ka_{k+1}a_{k+2}a_{k+3}a_{k+4}= 21111$ can not appear in $\{a_n\}_n$. 
\end{lem}

\begin{proof}
Since $\mathrm{rep}(x) > r_3 > \varphi$,  $x$ satisfies the case $[2]_k$ for all sufficiently large $k$ by Lemma \ref{buglem:7.4}. Then, by Lemma \ref{buglem:7.5}, $a_k$ must be equal to $1$ or $2$.   Furthermore, by Lemma \ref{buglem:7.2},  
\begin{gather*}
|W_{k+j}| = |W_{k+j-1}M_{k+j-2}| = |W_{k+j-1}| + q_{k+j-2}
\end{gather*}
is satisfied for all sufficiently large $k$ and all $j \ge 0$. Then $t_{k+j}$ is represented as 
\begin{gather}\label{eq:t}
t_{k+j} = \frac{|W_{k+j}|}{q_{k+j}} = \frac{|W_{k+j-1}| + q_{k+j-2}}{q_{k+j}} = \eta_{k+j} (t_{k+j-1}  + \eta_{k+j-1}) .
\end{gather}

Suppose that $a_ka_{k+1}a_{k+2}a_{k+3}a_{k+4} = 21111$ appears for some sufficiently large $k$. 
Then the recurrence relation of $q_n$ gives 
\begin{gather*}
 q_{k+3}=8q_{k-1}+3q_{k-2}, \qquad q_{k+4}=13q_{k-1}+5q_{k-2}, 
\end{gather*}
and hence 
\begin{gather*}
	\eta_{k+4}=\dfrac{q_{k+3}}{q_{k+4}}=\dfrac{8+3\eta_{k-1}}{13+5\eta_{k-1}} .
\end{gather*}
By (\ref{eq:t}), we have
\begin{gather*}
t_{k+4}= \dfrac{t_{k-1}+11+5\eta_{k-1}}{13+5\eta_{k-1}} . 
\end{gather*}
From Lemma \ref{buglem:7.6}(2) and $\rep(x)>r_3$, 
it follows 
\begin{gather*} 
	1+\dfrac{t_{k-1}+\eta_{k-1}}{1+\eta_{k-1}} > \dfrac{2(1869+2\varphi)}{2277} ,
\end{gather*}
and hence
\begin{gather}
	t_{k-1} > \dfrac{1461+4\varphi}{2277}(1+\eta_{k-1})-\eta_{k-1} .  \label{eq:matheq2}
\end{gather}
By using Lemma \ref{buglem:7.6}(1), (\ref{eq:matheq2}) and $\eta_{k-1} \geq 0$, we obtain
\begin{align*}
	\dfrac{r(|W_{k+4}|+q_{k+4}+q_{k+3}-2, x)}{|W_{k+4}|+q_{k+4}+q_{k+3}-2}
		&<1+\dfrac{1+\eta_{k+4}}{t_{k+4}+1+\eta_{k+4}}+\varepsilon_k \\
		&=1+\dfrac{21+8\eta_{k-1}}{t_{k-1}+32+13\eta_{k-1}}+\varepsilon_k \\
		&<1+\dfrac{21+8\eta_{k-1}}
			{\frac{1461+4\varphi}{2277}(1+\eta_{k-1})-\eta_{k-1}+32+13\eta_{k-1}}+\varepsilon_k \\
		&\leq 1+\dfrac{47817}{74325+4\varphi}+\varepsilon_k =1.64329\cdots +\varepsilon_k .
\end{align*}
Since $k$ is sufficiently large, this contradicts to $\rep(x)> r_3$.
\end{proof}

\begin{lem} \label{lem:1645}
	Let $x$ be a Sturmian word and 
	$\theta = [0, a_1, a_2, \cdots]$ be the slope of $x$. 
	If $\rep(x)>1.645$ and $k$ is sufficiently large, 
	then $a_k \in \{1,2\}$ and the following sequences can not appear  in $\{a_n\}_n$. 
	\begin{enumerate}
		\item $a_{k+j}= 1$ for all $j \geq 0$. 
		\item $a_ka_{k+1}=22$.  
		\item $a_ka_{k+1}a_{k+2}a_{k+3} a_{k+4}a_{k+5}=121212$, and then
			$a_ka_{k+1}a_{k+2}a_{k+3}a_{k+4}a_{k+5}=212121$. 
		\item $a_k a_{k+1} a_{k+2} a_{k+3} a_{k+4} a_{k+5} a_{k+6}=1211212$, and then
						$a_k a_{k+1} a_{k+2} a_{k+3} a_{k+4} a_{k+5}=211212$ . 
	\end{enumerate}
\end{lem}

\begin{proof}
By Lemma \ref{lem:21111}, we have $a_k \in \{1,2\}$.   
Both (1) and (2) follow from the same argument as \cite[below Proof of Lemma 7.6]{bugeaud}. 

(3)\; Suppose that $a_ka_{k+1}a_{k+2}a_{k+3} a_{k+4}a_{k+5}=121212$ appears for some sufficiently large $k$.
	Then the recurrence relation of $q_n$ gives  
	\begin{gather*} 
	\eta_{k+5}=\dfrac{q_{k+4}}{q_{k+5}}= \dfrac{15q_{k-1}+11q_{k-2}}{41q_{k-1}+30q_{k-2}} = \dfrac{15+11\eta_{k-1}}{41+30\eta_{k-1}},
	\end{gather*} 
	and by   (\ref{eq:t}) 
	\begin{gather*}
		t_{k+5}=\dfrac{|W_{k+5}|}{q_{k+5}}
			=\dfrac{t_{k-1}+20+15\eta_{k-1}}{41+30\eta_{k-1}} . 
	\end{gather*}
	From Lemma \ref{buglem:7.6}(1) and $\rep(x)>1.645$, 
	it follows 
	\begin{gather*}
		1+\dfrac{1+\eta_{k-1}}{t_{k-1}+1+\eta_{k-1}} > 1.645 ,
	\end{gather*}
	and hence
	\begin{gather}
		t_{k-1} < \dfrac{71}{129}(1+\eta_{k-1}) . \label{eq:matheq1}
	\end{gather}
	By using Lemma \ref{buglem:7.6}(2), (\ref{eq:matheq1}) and $\eta_{k-1} \leq 1$, we obtain 
	\begin{align*}
		\dfrac{r(q_{k+5}+q_{k+4}-2, x)}{q_{k+5}+q_{k+4}-2}
			&<1+\dfrac{t_{k+5}+\eta_{k+5}}{1+\eta_{k+5}}+\varepsilon_k \\
			&=1+\dfrac{t_{k-1}+35+26\eta_{k-1}}{56+41\eta_{k-1}}+\varepsilon_k \\
			&<1+\dfrac{\frac{71}{129}(1+\eta_{k-1})+35+26\eta_{k-1}}{56+41\eta_{k-1}}+\varepsilon_k \\
			&= \dfrac{11810+8714\eta_{k-1}}{7224+5289\eta_{k-1}}+\varepsilon_k \\
			&\leq \dfrac{11810 + 8714}{7224 + 5289}+\varepsilon_k =1.6402\cdots+\varepsilon_k  . 
	\end{align*}
	Since $k$ is sufficiently large, this contradicts to $\rep(x)>1.645$. If $a_ka_{k+1}a_{k+2}a_{k+3} a_{k+4}a_{k+5}=212121$ appears, then   $a_{k-1}a_ka_{k+1}a_{k+2}a_{k+3} a_{k+4}a_{k+5}$ must be equal to $1212121$ by (2), but $a_{k-1}a_ka_{k+1}a_{k+2}a_{k+3} a_{k+4} = 121212$ is impossible for sufficiently large $k$. 

(4)\; Suppose that $a_k a_{k+1} a_{k+2} a_{k+3} a_{k+4} a_{k+5} a_{k+6}=1211212$ appears  for some sufficiently large $k$.
	Then, by the same argument as (3), we have 
	\begin{gather*}
	q_{k+5}=25q_{k-1}+18q_{k-2}, \qquad q_{k+6}=68q_{k-1}+49q_{k-2}, \\
	\eta_{k+6}=\dfrac{q_{k+5}}{q_{k+6}}=\dfrac{25+18\eta_{k-1}}{68+49\eta_{k-1}} ,\quad
	t_{k+6}=\dfrac{|W_{k+6}|}{q_{k+6}}
			=\dfrac{t_{k-1}+34+25\eta_{k-1}}{68+49\eta_{k-1}} ,
	\end{gather*}
and hence	
	\begin{gather*}
		\dfrac{r(q_{k+6}+q_{k+5}-2, x)}{q_{k+6}+q_{k+5}-2}
			<1+\dfrac{t_{k+6}+\eta_{k+6}}{1+\eta_{k+6}}+\varepsilon_k 
			< \dfrac{1697}{1032}+\varepsilon_k 
			=1.6443\cdots+\varepsilon_k 
			<1.645.
	\end{gather*}
	This leads us to a contradiction.  Combining this with (2), the sequence $a_k a_{k+1} a_{k+2} a_{k+3} a_{k+4} a_{k+5}=211212$ is not also admitted. 
\end{proof}

\begin{lem} \label{lem:r2}
	Let $x$ be a Sturmian word such that 
	the slope $\theta$ of $x$ is equal to 
	\begin{gather*}
		[0, a_1, a_2, \cdots, a_K, \overline{2,1,1,2,1,1,1}]
	\end{gather*}
	for some $K$. 
	Assume that the case $[2]_k$ holds for all  $k \geq K$. Then $x$ satisfies 
	\begin{gather*}
		\rep(x) = r_2 = \dfrac{415\sqrt{149}-2693}{1438} = 1.65001\cdots. 
	\end{gather*}
\end{lem}

\begin{proof}
By Lemma \ref{buglem:7.6}(1), the following inequality is satisfied for sufficiently large $k$: 
\begin{gather} \label{eq:r_7k+K}
	\dfrac{r(|W_{7k+K}|+q_{7k+K}+q_{7k-1+K}-2, x)}{|W_{7k+K}|+q_{7k+K}+q_{7k-1+K}-2}
			<1+\dfrac{1+\eta_{7k+K}}{t_{7k+K}+1+\eta_{7k+K}}+\varepsilon_{7k+K}. 
\end{gather}
Since 
\begin{gather*}
	\eta_k=\dfrac{q_{k-1}}{q_k}=[0, a_k, a_{k-1}, \cdots, a_1],
\end{gather*}
we can compute the limit of $\eta_{7k+i +K}$ as $k \to \infty$ for $i = 0, 1, \cdots, 6$. It is easy to check  
\begin{align*}
	\displaystyle \lim_{k \to \infty} \eta_{7k+K} & =[0, \overline{1,1,1,2,1,1,2}] = \dfrac{5\sqrt{149} - 37}{38}, \\		
	\displaystyle \lim_{k \to \infty} \eta_{7k+1+K} &=[0, \overline{2,1,1,1,2,1,1}] = \dfrac{5\sqrt{149} - 39}{58}, \\
	  & \qquad \qquad \vdots
\end{align*}
and so on. By (\ref{eq:t}), $t_k$ is represented as 
\begin{gather*}
	t_k=\eta_k \eta_{k-1}+\eta_k \eta_{k-1} \eta_{k-2}+\cdots+\eta_k \eta_{k-1} \cdots \eta_{K+1}\eta_K
			+\dfrac{|W_K|}{q_k}.
\end{gather*}
This allows us to compute the limit of $t_{7k+i+K}$ as $k \to \infty$  for $i = 0, 1, \cdots, 6$. 
We sketch how to compute $t_{7k+K}$ as an example. Let 
\begin{gather*}
s_k = \sum_{\ell = 0}^{7k-1} \prod_{j = \ell}^{7k} \eta_{j + K}, \\
\alpha_k = \sum_{\ell = 0}^{6} \prod_{j = \ell}^{7} \eta_{7k-7 + j + K}, \qquad \beta_k = \prod_{j=1}^7 \eta_{7k - 7 + j + K}.
\end{gather*}
Then the limits of $\eta_{7k+i+K}$ give explicit values of 
\begin{gather*}
 \alpha = \lim_{k \to \infty} \alpha_k =  \sum_{\ell = 0}^{6} \prod_{j = \ell}^{7} \lim_{k \to \infty} \eta_{7k-7 + j + K}, \qquad \beta = \lim_{k \to \infty} \beta_k = \prod_{j=1}^7 \lim_{k \to \infty} \eta_{7k - 7 + j + K}.
\end{gather*}
Since $\eta_k$ for $k >  K+1$ satisfies 
\begin{gather*}
 \eta_k = [0, a_k, a_{k-1}, \cdots, a_K, \cdots,  a_1] \le  \dfrac{1}{1 + \dfrac{1}{2 + 1 }}  = \dfrac{3}{4} ,
\end{gather*}
the sequence  $\{s_k\}_k$ is bounded. Then the relation $s_k = \alpha_k + \beta_k s_{k-1} $ implies 
\begin{gather*}
\frac{\alpha}{1- \beta} \le \liminf_{k \to \infty} s_k \le \limsup_{k \to \infty} s_k \le \frac{\alpha}{1- \beta}. 
\end{gather*} 
This gives 
\begin{gather*}
\lim_{k \to \infty} t_{7k+K} = \lim_{k \to \infty} \left ( s_k + \frac{|W_K|}{q_k} \right ) = \lim_{k \to \infty} s_k =  \frac{\alpha}{1 - \beta}.
\end{gather*}
By this method, we obtain
\begin{gather*}
	\displaystyle \lim_{k \to \infty} t_{7k+K}=\dfrac{1568-45\sqrt{149}}{1159}, \quad
	\displaystyle \lim_{k \to \infty} t_{7k+1+K}=\dfrac{3313-105\sqrt{149}}{3538}, \quad \cdots 
\end{gather*}
and so on. Then (\ref{eq:r_7k+K}) implies
\begin{gather*}
	\rep(x) \leq \displaystyle \lim_{k \to \infty} \left( 1+\dfrac{1+\eta_{7k+K}}{t_{7k+K}+1+\eta_{7k+K}} \right) = r_2. 
\end{gather*}
Similarly as in \cite[Proof of Theorem 3.4]{bugeaud}, 
$r(n, x)$ satisfies the following inequality: 
\begin{align} \label{eq:rmin}
	r(n, x)
	\geq \begin{cases}
		n+q_k+q_{k-1} & (q_k+q_{k-1}-1 \leq n \leq |W_k|+q_k+q_{k-1}-2), \\
		n+|W_k|+q_k+q_{k-1} & (|W_k|+q_k+q_{k-1}-1 \leq n \leq q_{k+1}+q_k-2). 
	\end{cases}
\end{align}
We set 
\begin{align*}
\rho_j&=\displaystyle \liminf_{k \to \infty} 
		\dfrac{|W_{7k+j+K}|+2q_{7k+j+K}+2q_{7k+j+K-1}-2}{|W_{7k+j+K}|+q_{7k+j+K}+q_{7k+j+K-1}-2} \\
	&=1+\displaystyle \liminf_{k \to \infty} \dfrac{1+\eta_{7k+j+K}}{t_{7k+j+K}+1+\eta_{7k+j+K}}, \\
\rho'_j&=\displaystyle \liminf_{k \to \infty} 
		\dfrac{|W_{7k+j+K}|+q_{7k+j+K+1}+2q_{7k+j+K}+q_{7k+j+K-1}-2}{q_{7k+j+K+1}+q_{7k+j+K}-2} \\
	&=1+\displaystyle \liminf_{k \to \infty} \dfrac{t_{7k+j+1+K}+\eta_{7k+j+1+K}}{1+\eta_{7k+j+1+K}}. 
\end{align*}
for $j=0, 1, \cdots, 6$. 
Since $	\min(\rho_0, \cdots, \rho_6) = \rho_0=r_2$ as a consequence of computations, 
we obtain 
\begin{gather*}
	\displaystyle \liminf_{k \to \infty} \dfrac{|W_k|+2q_k+2q_{k-1}-2}{|W_k|+q_k+q_{k-1}-2}
		=r_2. 
\end{gather*}
Similarly, since 
\begin{gather*}
	\min(\rho'_0, \cdots, \rho'_6) = \rho'_3 = \dfrac{7(85-\sqrt{149})}{305}, 
\end{gather*}
we obtain
\begin{gather*}
	\displaystyle \liminf_{k \to \infty} \dfrac{|W_k|+q_{k+1}+2q_k+q_{k-1}-2}{q_{k+1}+q_{k-1}-2}
		= \rho'_3
		>r_2. 
\end{gather*}
Then (\ref{eq:rmin}) gives $\rep(x) \geq r_2$. 
This concludes $\rep(x)=r_2$. 
\end{proof}

\begin{lem} \label{lem:1650}
	Let $x$ be a Sturmian word and 
	$\theta = [0, a_1, a_2, \cdots]$ be  the slope of $x$. 
	If $\rep(x) \geq r_2$ and $k$ is sufficiently large, then $a_k \in \{1,2\}$ and the following sequences can not appear  in $\{a_n\}_n$. 
	\begin{enumerate}
		\item $a_k a_{k+1} a_{k+2} a_{k+3} a_{k+4} a_{k+5} a_{k+6} a_{k+7}=21112111$. 
		\item $a_ka_{k+1} \cdots  a_{k+11} a_{k+12} =2111212111212$. 
		\item $a_k a_{k+1} \cdots a_{k+8} a_{k+9}=2112111212$. 
		\item $a_k a_{k+1} a_{k+2} a_{k+3}=1212$. 
	\end{enumerate}
\end{lem}

\begin{proof}
By Lemmas \ref{buglem:7.4} and \ref{lem:21111}, $a_k \in \{1,2\}$ and the case $[2]_k$ holds for all  sufficiently large $k$.   Since (1), (2) and (3) are proved by the same argument as in Proof of Lemmas \ref{lem:21111} and \ref{lem:1645}, we  explain only an outline. 

(1)\; Suppose  that  $a_k a_{k+1} \cdots a_{k+6} a_{k+7}=21112111$ appears  for some sufficiently large $k$.
	Then we have 
	\begin{gather*}
	q_{k+6}=50q_{k-1}+19q_{k-2}, \qquad q_{k+7}=79q_{k-1}+30q_{k-2},
    \end{gather*}	
	and hence 
	\begin{gather*}
		\eta_{k+7}=\dfrac{q_{k+6}}{q_{k+7}}=\dfrac{50+19\eta_{k-1}}{79+30\eta_{k-1}} ,\quad 
		t_{k+7}=\dfrac{|W_{k+7}|}{q_{k+7}}
			=\dfrac{t_{k-1}+69+27\eta_{k-1}}{79+30\eta_{k-1}} . 
	\end{gather*}
	These imply 
	\begin{align*}
		\dfrac{r(|W_{k+7}|+q_{k+7}+q_{k+6}-2, x)}{|W_{k+7}|+q_{k+7}+q_{k+6}-2}
			&<1+\dfrac{1+\eta_{k+7}}{t_{k+7}+1+\eta_{k+7}}+\varepsilon_k \\
			&< 1+\dfrac{185502}{280593+415\sqrt{149}}+\varepsilon_k  =1.64938\cdots +\varepsilon_k .
	\end{align*}
	This contradicts to $\rep(x) \geq r_2$. 

(2)\; Suppose  that $a_ka_{k+1} \cdots  a_{k+11} a_{k+12} =2111212111212$ appears for some sufficiently large $k$.
	Then we have 
	\begin{gather*}
	q_{k+11}=1072q_{k-1}+407q_{k-2}, \qquad q_{k+12}=2921q_{k-1}+1109q_{k-2},
	\end{gather*}	
	and hence 
	\begin{gather*}
		\eta_{k+12}=\dfrac{q_{k+11}}{q_{k+12}}=\dfrac{1072+407\eta_{k-1}}{2921+1109\eta_{k-1}} ,\quad
		t_{k+12}=\dfrac{|W_{k+12}|}{q_{k+12}}
			=\dfrac{t_{k-1}+1515+576\eta_{k-1}}{2921+1109\eta_{k-1}} . 
	\end{gather*}
	These imply 
	\begin{align*}
		\dfrac{r(q_{k+12}+q_{k+11}-2, x)}{q_{k+12}+q_{k+11}-2}
			&<1+\dfrac{t_{k+12}+\eta_{k+12}}{1+\eta_{k+12}}+\varepsilon_k \\
			&< \dfrac{27135284+415\sqrt{149}}{16466401}+\varepsilon_k =1.6482\cdots +\varepsilon_k . 
	\end{align*}
	This contradicts to $\rep(x) \geq r_2$. 

(3)\; Suppose that $a_k a_{k+1} \cdots a_{k+8} a_{k+9}=2112111212$ appears for some sufficiently large $k$.
	Then we have 
	\begin{gather*}
	q_{k+8}=178q_{k-1}+69q_{k-2}, \qquad q_{k+9}=485q_{k-1}+188q_{k-2},
    \end{gather*}	
	and hence 
	\begin{gather*}
		\eta_{k+9}=\dfrac{q_{k+8}}{q_{k+9}}=\dfrac{178+69\eta_{k-1}}{485+188\eta_{k-1}} ,\quad
		t_{k+9}=\dfrac{|W_{k+9}|}{q_{k+9}}
			=\dfrac{t_{k-1}+251+98\eta_{k-1}}{485+188\eta_{k-1}} . 
	\end{gather*}
	Therefore,  
	\begin{align*}
		\dfrac{r(q_{k+9}+q_{k+8}-2, x)}{q_{k+9}+q_{k+8}-2}
			&<1+\dfrac{t_{k+9}+\eta_{k+9}}{1+\eta_{k+9}}+\varepsilon_k \\
			&< \dfrac{4529477+415\sqrt{149}}{2749880}+\varepsilon_k 
			=1.6489\cdots +\varepsilon_k . 
	\end{align*}
This contradicts to $\rep(x) \geq r_2$. 

(4)\; Suppose that $a_{k}a_{k+1}a_{k+2}a_{k+3} = 1212$ appears for 
    some sufficiently large $k$. 
	We may assume that $2$ occurs at least twice before $a_{k}$ in $\theta$. 
	Since  three sequences $221212$, $121212$ and $211212$ can not appear  by Lemma \ref{lem:1645}, 
	the only possible case is $a_{k-2} \cdots a_{k+3} = 111212$.  
	Then $w = a_{k-4} \cdots a_{k+3} = 12111212$ must appear because that both $1111212$ and $22111212$ does not appear
	by Lemmas  \ref{lem:21111} and \ref{lem:1645}. 
	From Lemmas \ref{lem:21111}, \ref{lem:1645}, (1), (2) and (3), 
	any of $21w, 211w, 111w$ and $212w, 2112w, 21112w, 11112w$ can not appear  for sufficiently large $k$. 
	This implies that both $1w$ and $2w$ are not admitted in $\{a_n\}_n$ for sufficiently large $k$.  
	This is a contradiction. 
\end{proof}

\begin{prop}\label{prop:2111r_1}	
	Let $x$ be a Sturmian word and 
	$\theta = [0, a_1, a_2, \cdots]$ be the slope of $x$. 
	If
	\begin{gather*}
		\rep(x) > r_1 = \dfrac{48+\sqrt{10}}{31} = 1.65039\cdots
	\end{gather*}
	and $k$ is sufficiently large,
	then 
	the sequence $a_k a_{k+1} a_{k+2}a_{k+3}=2111$ can not appear in $\{a_n\}_n$. 
	In particular, 
	the slope $\theta$ is of the form $[0, a_1, a_2, \cdots, a_K, \overline{2, 1, 1}]$ for some $K$.
	 
\end{prop}

\begin{proof}
Since $\rep(x)>r_1$, 
there exists $\varepsilon>0$ such that $\rep(x)>r_1+\varepsilon$. We choose $k_0 > K$ satisfying  $2/q_{k_0} < \varepsilon$. 
By the assumption $\rep(x) >r_1$, $a_k \in \{1,2\}$ and the case $[2]_k$ holds for  all sufficiently large $k$.   
By Lemmas \ref{lem:21111}, \ref{lem:1645} and \ref{lem:1650}, 
 possible sequences  in $\{a_n\}_n$ are only 
\begin{gather*}
	a_{k} \cdots a_{k+3} = 2112 \qquad \text{and} \qquad a_k \cdots a_{k+4} = 21112
\end{gather*}
when $k$ is sufficiently large, 
and moreover the sequence 
\begin{gather*}
	a_{k} \cdots a_{k+7} = 21112111
\end{gather*}
can not appear. 
Therefore, it is sufficient to prove  that the sequence 
\begin{gather*}
	a_k \cdots a_{k+3n+4} = 2111(211)^n1
\end{gather*}
does not appear for every $n \geq 2$.  We prove this by contradiction. Fix a positive integer $n$ and suppose that the sequence 
\begin{gather*}
	a_k \cdots a_{k+3n+4} = 2111(211)^n1
\end{gather*}
appears for some sufficiently large $k \ge k_0$. 
For each positive integer $\ell$, we set $u_{\ell} = q_{k+3\ell}$ and $v_{\ell} = q_{k + 3\ell - 1}$. 
From the recurrence relation of $q_n$, it follows 
\begin{gather*}
	\begin{pmatrix}
		u_{\ell+1} \\
		v_{\ell+1}
	\end{pmatrix}
	= \begin{pmatrix}
		5 & 2 \\
		3 & 1
	\end{pmatrix}
	\begin{pmatrix}
		u_\ell \\ 
		v_\ell
	\end{pmatrix}, \qquad 
	\begin{pmatrix}
		u_1 \\
		v_1
	\end{pmatrix}
	= \begin{pmatrix}
		q_{k+3} \\
		q_{k+2}
	\end{pmatrix}
	= \begin{pmatrix}
		8q_{k-1} + 3q_{k-2} \\
		5q_{k-1} + 2q_{k-2}
	\end{pmatrix} . 
\end{gather*}
By solving this recurrence relation, $(u_\ell, v_\ell)$ is described as 
\begin{align*}
	u_\ell &= \dfrac{1}{2\sqrt{10}} \left( ((-26+8\sqrt{10})p^{\ell-1}+(26+8\sqrt{10})q^{\ell-1})q_{k-1} \right. \\
		&\qquad \left. +((-10+3\sqrt{10})p^{\ell-1}+(10+3\sqrt{10})q^{\ell-1})q_{k-2} \right) \\
	v_\ell &= \dfrac{1}{2\sqrt{10}} \left( ((-14+5\sqrt{10})p^{\ell-1}+(14+5\sqrt{10})q^{\ell-1})q_{k-1} \right. \\
		&\qquad \left. +((-5+2\sqrt{10})p^{\ell-1}+(5+2\sqrt{10})q^{\ell-1})q_{k-2} \right)
\end{align*}
where $p=3-\sqrt{10}$ and $q=3+\sqrt{10}$.
Hence we obtain explicit formulas of $q_{k+3n+3}=u_{n+1}$, $q_{k+3n+4}=u_{n+1}+v_{n+1}$ and $\eta_{k+3n+4} = q_{k+3n+3}/q_{k+3n+4}$. 
Moreover, since
\begin{gather*}
	q_{k+3\ell+1}+q_{k+3\ell+2}+q_{k+3\ell+3}
		=2u_\ell+v_\ell+v_{\ell+1}+u_{\ell+1}
		=10u_\ell+4v_\ell
\end{gather*}
and 
\begin{gather*}
q_{k-2}+q_{k-1}+q_k+q_{k+1}+q_{k+2}+q_{k+3} = 19q_{k-1}+8q_{k-2},
\end{gather*}
we have
\begin{align}
	\begin{aligned} \label{eq:sum}
	&q_{k-2}+q_{k-1}+q_k+q_{k+1}+\cdots+q_{k+3n+2} \\
		=&19q_{k-1}+8q_{k-2}+\displaystyle \sum^n_{\ell=1} (10u_\ell+4v_\ell) -u_{n+1} \\
		=&\dfrac{1}{30}\left \{ (-10 + (170-53\sqrt{10})p^n + (170 +53\sqrt{10}) q^n) q_{k-1}  \right . \\
			& \qquad \left . -5 \left (-4 + (-13 +4\sqrt{10})p^n -(13+4\sqrt{10})q^n )q_{k-2} \right ) \right \} .
			\end{aligned}
\end{align}
The assumption $\rep(x) > r_1$ and 
Lemma \ref{buglem:7.6}(2)
give
\begin{gather*}
	t_{k-1}>\dfrac{17+\sqrt{10}}{31}(1+\eta_{k-1})-\eta_{k-1} .
\end{gather*}
Since
\begin{gather*}
	t_{k+3n+4}=\dfrac{|W_{k-1}|+q_{k-2}+q_{k-1}+q_k+q_{k+1}+\cdots+q_{k+3n+2}}{q_{k+3n+4}} 
\end{gather*}
by  (\ref{eq:t}), 
Lemma \ref{buglem:7.6}(1) and (\ref{eq:sum}) lead us to
\begin{align*}
	\dfrac{r(|W_{k+3n+4}|+q_{k+3n+4}+q_{k+3n+3}-2, x)}{|W_{k+3n+4}|+q_{k+3n+4}+q_{k+3n+3}-2} 
		&<1+\dfrac{1+\eta_{k+3n+4}}{t_{k+3n+4}+1+\eta_{k+3n+4}}+\varepsilon_{k+3n+4} \\
		& < 1 + \dfrac{1+\eta_{k+3n+4}}{\dfrac{17+\sqrt{10}}{31}(1+\eta_{k+3n+4})+1}+\varepsilon_{k+3n+4} \\
		&= 1+\dfrac{A}{B} +\varepsilon_{k+3n+4},
\end{align*}
where $A$ and $B$ are given by 
\begin{align*}
	A&=93 \left \{ (-66+21\sqrt{10})p^n+(66+21\sqrt{10})q^n	+((-25+8\sqrt{10})p^n+(25+8\sqrt{10})q^n))\eta_{k-1} \right \} , \\
	B&=60+40\sqrt{10}+31(-304+97\sqrt{10})p^n+31(304+97\sqrt{10})q^n \\
			&\qquad +\left \{60+40\sqrt{10}+31(-115+37\sqrt{10})p^n+31(115+37\sqrt{10})q^n \right \} \eta_{k-1} .
\end{align*}
Since $A/B$ is monotonically decreasing with respect to $\eta_{k-1} \geq 0 $, 
we obtain 
\begin{gather*}
	\dfrac{A}{B}
	 \leq \dfrac{93((-66+21\sqrt{10})p^n+(66+21\sqrt{10})q^n)}
			{60+40\sqrt{10}+31(-304+97\sqrt{10})p^n+31(304+97\sqrt{10})q^n} .
		\end{gather*}
Since the right-hand side is monotonically increasing with respect to $n$ and 
\begin{gather*}
\displaystyle \lim_{n \to \infty} \left( 1+\dfrac{279((-22+7\sqrt{10})p^n+(22+7\sqrt{10})q^n)}
			{60+40\sqrt{10}+31(-304+97\sqrt{10})p^n+31(304+97\sqrt{10})q^n} \right)
	=r_1 , 
\end{gather*}
we conclude
\begin{gather*}
	\dfrac{r(|W_{k+3n+4}|+q_{k+3n+4}+q_{k+3n+3}-2, x)}{|W_{k+3n+4}|+q_{k+3n+4}+q_{k+3n+3}-2}
		<r_1+\varepsilon_{k+3n+4} . 
\end{gather*}
This contradicts to $r_1+\varepsilon < \rep(x)$ since $\varepsilon_{k+3n+4} = 2/q_{k+3n+4}  < 2/q_{k_0} <  \varepsilon$.
Therefore, for every $n$, the sequence $2111 (211)^n 1$ can not appear. 
This means that, if $\rep(x) > r_1$ and $k$ is sufficiently large, 
then
the sequence $a_ka_{k+1}a_{k+2}a_{k+3}=2111$ can not appear.  
Consequently,  $\theta$ is of the form 
$[0, a_1, a_2, \cdots, a_K, \overline{2, 1, 1}]$
for some integer $K$. 
\end{proof}

\begin{thm} \label{th:repmax2}
	There is no Sturmian word $x$ such that
	$r_1 < \rep(x) < r_{\max}$. 
	If a Sturmian word $x$ satisfies
	$r_2 \leq \rep(x) \leq r_1$, 
	then the continued fraction expansion of the slope of $x$ is of the form
	\begin{gather*}
		[0, a_1, a_2, \cdots, a_K, (2,1,1)^{n_1}, 1, (2,1,1)^{n_2}, 1, \cdots]
	\end{gather*}
	for some $K$ and some sequence $\{ n_i \}_i$ of integers greater than $1$. 
\end{thm}

\begin{proof}
Let $x$ be a Sturmian word of slope $\theta = [0,a_1, a_2, \cdots]$. 
By Proposition \ref{prop:2111r_1} and \cite[Proof of Theorem 3.4]{bugeaud}, 
$\rep(x)$ is equal to $r_{\max}$ when $\rep(x) > r_1$.  
If $\rep(x) \geq r_2$ and $k$ is sufficiently large, then only a sequence of the form 
$ a_k \cdots a_{k + 3n}= (211)^{n}1 $
for some $n \geq 2$ is admitted to $\theta$ by Lemmas \ref{lem:21111}, \ref{lem:1645} and \ref{lem:1650}. 
\end{proof}

\vskip 10mm
\section{Auxiliary lemmas of some continued fractions} \label{sec:continuedfraction}

For a positive integer $n$, we set
\begin{alignat*}{2}
	& e^{(n)}_0(m):= [0, \overline{(1,1,2)^m,1,(1,1,2)^{n-m+1}}] &\quad & (m=0, 1, 2, \cdots, n+1) \\
	& e^{(n)}_1(m):= [0, \overline{2,(1,1,2)^m,1,(1,1,2)^{n-m},1,1}] &\quad & (m=0, 1, 2, \cdots, n) \\
	& e^{(n)}_2(m):= [0, \overline{1,2,(1,1,2)^m,1,(1,1,2)^{n-m},1}] &\quad & (m=0, 1, 2, \cdots, n) . 
\end{alignat*}
These continued fractions will be used in \S \ref{sec:accumulate} and \S \ref{sec:largest}. 
In this section, we investigate the double sequence $\left\{ e^{(n)}_i(m) \right\}_{n, m}$ for $i=0, 1, 2$. 
Since 
\begin{gather*}
	e^{(n)}_1(m)=\dfrac{1}{2+e^{(n)}_0(m)}, \quad 
	e^{(n)}_2(m)=\dfrac{1}{1+e^{(n)}_1(m)}, \quad 
	e^{(n)}_0(m+1)=\dfrac{1}{1+e^{(n)}_2(m)} , 
\end{gather*}
we have  the following relations:
\begin{align}
	\begin{aligned} \label{eq:e-rec}
	& e^{(n)}_0(m+1)=\dfrac{1}{1+\dfrac{1}{1+\dfrac{1}{2+e^{(n)}_0(m)}}}=[0, 1, 1, 2+e^{(n)}_0(m)] , \\
	& e^{(n)}_1(m+1)=\dfrac{1}{2+\dfrac{1}{1+\dfrac{1}{1+e^{(n)}_1(m)}}}=[0, 2, 1, 1+e^{(n)}_1(m)] , \\
	& e^{(n)}_2(m+1)=\dfrac{1}{1+\dfrac{1}{2+\dfrac{1}{1+e^{(n)}_2(m)}}}=[0, 1, 2, 1+e^{(n)}_2(m)]  
	\end{aligned}
\end{align}
and 
\begin{gather}\label{eq:e-relation}
1 + e^{(n)}_2(m) + e^{(n)}_2(m)e^{(n)}_1(m) = 2. 
\end{gather}
Solving  (\ref{eq:e-rec}), 
$e^{(n)}_i(m)$ is represented as 
\begin{gather} \label{eq:e-gen}
	e^{(n)}_i(m) = \dfrac{a_i(m) e^{(n)}_i(0)+b_i(m)}{c_i(m) e^{(n)}_i(0)+d_i(m)} 
\end{gather}
for $i = 0,1,2$, 
where
\begin{align*}
	A_0(m) &= \begin{pmatrix}
				a_0(m) & b_0(m) \\
				c_0(m) & d_0(m) \\
			\end{pmatrix}
			=\begin{pmatrix}
				-2+\sqrt{10}+(2+\sqrt{10}) r^m & 3(1-r^m) \\
				2(1-r^m) & 2+\sqrt{10}+(-2+\sqrt{10}) r^m \\
			\end{pmatrix}, \\
	A_1(m) &= \begin{pmatrix}
				a_1(m) & b_1(m) \\
				c_1(m) & d_1(m) \\
			\end{pmatrix}
			=\begin{pmatrix}
				-2+\sqrt{10}+(2+\sqrt{10}) r^m & 2(1-r^m) \\
				3(1-r^m) & 2+\sqrt{10}+(-2+\sqrt{10}) r^m \\
			\end{pmatrix}, \\
	A_2(m) &= \begin{pmatrix}
				a_2(m) & b_2(m) \\
				c_2(m) & d_2(m) \\
			\end{pmatrix}
			=\begin{pmatrix}
				-1+\sqrt{10}+(1+\sqrt{10}) r^m & 3(1-r^m) \\
				3(1-r^m) & 1+\sqrt{10}+(-1+\sqrt{10}) r^m \\
			\end{pmatrix}, 
\end{align*}
and $r= p/q=(3-\sqrt{10})/(3+\sqrt{10})$. 
Furthermore, by (\ref{eq:e-gen}), we can define $e^{(n)}_0(m)$ for $m \ge n+2$ and $e^{(n)}_1(m)$, $e^{(n)}_2(m)$ for $m \ge n+1$. Since $r^m \to 0$ as $m \to \infty$, there exist the limits $ \lim_{m \to \infty} e^{(n)}_i(m) $ for $i = 0,1,2$.

\begin{lem} \label{lem:e-property}
	$(1)$ For all $n \geq 1$ and all $m \geq 0$, $e^{(n)}_i(m), i = 0,1,2,$ are estimated as 
			\begin{gather*}
				0 < e^{(n)}_0(m)<\frac{3}{5}, \qquad 
				0 <  e^{(n)}_1(m)< \frac{2}{5}, \qquad 
				0< e^{(n)}_2(m)< \frac{3}{4} . 
			\end{gather*}
	
	$(2)$ $e^{(n)}_0(n+1), e^{(n)}_0(0), e^{(n)}_1(0)$ and $e^{(n)}_2(0)$ converge as $n \to \infty$ and 
			\begin{align*}
				 \displaystyle \lim_{n \to \infty} e^{(n)}_0(n+1) & = [0, \overline{1,1,2}]
					=\dfrac{-2+\sqrt{10}}{2}, \\
				e_0:= \displaystyle \lim_{n \to \infty} e^{(n)}_0(0) & = [0, 1, \overline{1,1,2}]
					=\dfrac{\sqrt{2}}{\sqrt{5}}, \\
				e_1:= \displaystyle \lim_{n \to \infty} e^{(n)}_1(0) & = [0, 2, 1,\overline{1,1,2}]
					=\dfrac{10-\sqrt{10}}{18}, \\
				e_2:= \displaystyle \lim_{n \to \infty} e^{(n)}_2(0) &= [0, 1, 2, 1,\overline{1,1,2}]
					=\dfrac{28+\sqrt{10}}{43} .
			\end{align*}
	
	$(3)$ For every $m$, there exist the following limits: 
		 	\begin{gather*}
				e_0(m) :=\displaystyle \lim_{n \to \infty} e^{(n)}_0(m), \qquad 
				e_1(m) :=\displaystyle \lim_{n \to \infty} e^{(n)}_1(m), \qquad 
				e_2(m) :=\displaystyle \lim_{n \to \infty} e^{(n)}_2(m) . 
			\end{gather*}
			These convergents are uniform with respect to $m$. 
			
	$(4)$ $e_0(m), e_1(m)$ and $e_2(m)$ converge as $m \to \infty$ and  
			\begin{gather*}
				\displaystyle \lim_{m \to \infty} e_0(m)=-1+\dfrac{\sqrt{5}}{\sqrt{2}} , \quad
				\displaystyle \lim_{m \to \infty} e_1(m)=\dfrac{-2+\sqrt{10}}{3} , \quad
				\displaystyle \lim_{m \to \infty} e_2(m)=\dfrac{-1+\sqrt{10}}{3} . 
			\end{gather*}
\end{lem}

\begin{proof}
(1)\; It is obvious that
	\begin{gather*}
		e^{(n)}_0(m+1)=[0, 1, 1, 2+e^{(n)}_0(m)]
			<[0, 1, 1, 2+0]
			=\dfrac{3}{5}.
	\end{gather*}
The remaining inequalities are proved by a similar fashion.  

(2)\; Since $e^{(n)}_0(n+1)=[0, \overline{(1,1,2)^{n+1},1}]$, 
	it is known by \cite[Lemma1.24(1)]{aigner} that   
	\begin{gather*}
		\left| e^{(n)}_0(n+1)-[0, \overline{1,1,2}] \right| \leq \dfrac{1}{2^{(3n+6)-2}}=\dfrac{1}{2^{3n+4}} . 
	\end{gather*}
		This gives $\displaystyle \lim_{n \to \infty} e^{(n)}_0(n+1)=[0, \overline{1,1,2}]=(-2+\sqrt{10})/2$. 
	Similarly,  $e_0, e_1$ and $e_2$ are computed. 

(3)\; We use the expression (\ref{eq:e-gen}).  
	The determinant of $A_0(m)$ is equal to $40 r^m$. 
	Since $|r^m|<1$ and $e^{(n)}_0(0)>0$, we have
	\begin{align*}
		c_0(m) e^{(n)}_0(0)+d_0(m)
			&=2 (1-r^m) e^{(n)}_0(0)+2+\sqrt{10}+(-2+\sqrt{10})r^m \\
			&>2 \cdot 0 +2+\sqrt{10}+(-2+\sqrt{10}) \cdot (-1) = 4 
	\end{align*}
	for all $m$, 
	and then 
	\begin{align*}
		\left| e_0(m) - e^{(n)}_0(m) \right|
			&=\left| \dfrac{\det{A_0(m)} \left( e_0-e^{(n)}_0(0) \right) }
				{ \left( c_0(m) e_0+d_0(m) \right) \left( c_0(m) e^{(n)}_0(0)+d_0(m) \right)} \right| \\
			&<\left|\det{A_0(m)} \left( e_0-e^{(n)}_0(0) \right) \right| \\
			&<40 \left| e_0-e^{(n)}_0(0) \right| . 
	\end{align*}
	This implies that $e^{(n)}_0(m)$ converges uniformly 
	with respect to $m$ as $n \to \infty$. 
	Similarly, 
	both $e^{(n)}_1(m)$ and $e^{(n)}_2(m)$ also converge uniformly with respect to $m$. 

(4)\;  It follows from (\ref{eq:e-gen}) that
	\begin{gather*}
		 \displaystyle \lim_{m \to \infty} e_0(m)
			=\dfrac{(-2+\sqrt{10})e_0+3}{2e_0+(2+\sqrt{10})}
			=-1+\dfrac{\sqrt{5}}{\sqrt{2}}.
	\end{gather*}
The limits of $e_1(m)$ and $e_2(m)$ are similarly computed. 
\end{proof}

\begin{lem} \label{lem:e-decrease}
	For all $n \geq 1$, 
	we have 
	\begin{align*}
		& e^{(n)}_1(0) < e^{(n)}_1(2) < e^{(n)}_1(4) < \cdots < \displaystyle \lim_{m \to \infty} e^{(n)}_1(m)
			< \cdots < e^{(n)}_1(5) < e^{(n)}_1(3) < e^{(n)}_1(1) \\
		<& e^{(n)}_0(1) < e^{(n)}_0(3) < e^{(n)}_0(5) < \cdots < \displaystyle \lim_{m \to \infty} e^{(n)}_0(m)
			< \cdots < e^{(n)}_0(4) < e^{(n)}_0(2) < e^{(n)}_0(0) \\
		<& e^{(n)}_2(1) < e^{(n)}_2(3) < e^{(n)}_2(5) < \cdots < \displaystyle \lim_{m \to \infty} e^{(n)}_2(m)
			< \cdots < e^{(n)}_2(4) < e^{(n)}_2(2) < e^{(n)}_2(0) . 
	\end{align*}
\end{lem}

\begin{proof}
By using (\ref{eq:e-gen}), we observe 
\begin{gather*}
	e^{(n)}_0(m)<e^{(n)}_0(m+1)  \iff r^m \left( 2 \left( e^{(n)}_0(0) \right)^2 +4 e^{(n)}_0(0) -3 \right) <0 . 
\end{gather*}
Since 
\begin{gather*}
	e^{(n)}_0(0)=[0, \overline{1, (1,1,2)^{n+1}}] > [0, \overline{1,1,2}]=\dfrac{-2+\sqrt{10}}{2} 
\end{gather*}
by \cite[Lemma1.24(2)]{aigner}, we have
$2 \left( e^{(n)}_0(0) \right)^2 +4 e^{(n)}_0(0) -3>0$, 
and hence, 
\begin{gather*}
 e^{(n)}_0(m)<e^{(n)}_0(m+1) \iff r^m <0.
\end{gather*} 
Similarly, we have
\begin{gather*}
e^{(n)}_0(m)<e^{(n)}_0(m+2) \iff r^m <0 .
\end{gather*} 
These show 
\begin{gather*}
	e^{(n)}_0(1) < e^{(n)}_0(3) < e^{(n)}_0(5) < \cdots < \displaystyle \lim_{m \to \infty} e^{(n)}_0(m)
		< \cdots < e^{(n)}_0(4) < e^{(n)}_0(2) < e^{(n)}_0(0) . 
\end{gather*}
The desired inequalities among $e^{(n)}_i(m)$ for $i = 1,2$ are proved by the same argument as above.  
Furthermore,  it follows from \cite[Lemma1.24(2)]{aigner} that
\begin{gather*}
	e^{(n)}_1(1)=[0, \overline{2, 1, 1, 2, 1, (1,1,2)^{n-1}, 1, 1}] < [0, \overline{1, 1, 2, 1, (1,1,2)^n}]=e^{(n)}_0(1)
\end{gather*}
and
\begin{gather*}
	e^{(n)}_0(0)=[0, \overline{1, (1,1,2)^{n+1}}] < [0, \overline{1, 2, 1, 1, 2, 1, (1,1,2)^{n-1}, 1}]=e^{(n)}_2(1) . 
\end{gather*}
This completes the proof. 
\end{proof}

\vskip 10mm
\section{Some sequence of Sturmian words} \label{sec:accumulate}

For $n \geq 1$, let $x^{(n)}$ be a Sturmian word such that 
the slope of $x^{(n)}$ is equal to 
\begin{gather*}
	\theta^{(n)} = [0, a_1, a_2, \cdots, a_K, \overline{(2, 1, 1)^n, 2,1,1,1}]
\end{gather*}
for some integer $K$. We assume $x^{(n)}$ satisfies the following condition. 
\begin{itemize}
\item The case $[2]_k$ in Lemma \ref{buglem:7.2} holds for all $k > K$. 
\end{itemize}
For simplicity, we assume that $K$ and positive integers $a_1, a_2, \cdots, a_K$  are independent of $n$. Let $p^{(n)}_j/q^{(n)}_j$ be the $j$-th convergent of $\theta^{(n)}$ and let 
\begin{gather*}
\eta^{(n)}_j = \dfrac{q^{(n)}_{j-1}}{q^{(n)}_j}, \qquad t^{(n)}_j = \dfrac{|W^{(n)}_j|}{q^{(n)}_j}, \qquad \varepsilon^{(n)}_j = \dfrac{2}{q^{(n)}_j}.
\end{gather*}  
We set 
\begin{align*}
	\zeta^{(n)}_{i}&:=\liminf_{k \to \infty} 
			\left( 1+\dfrac{1+\eta^{(n)}_{(3n+4)k+i+K}}{t^{(n)}_{(3n+4)k+i+K}+1+\eta^{(n)}_{(3n+4)k+i+K}} \right) , \\
	\xi^{(n)}_{i}&:=\liminf_{k \to \infty} 
			\left( 1+\dfrac{t^{(n)}_{(3n+4)k+i+K}+\eta^{(n)}_{(3n+4)k+i+K}}{1+\eta^{(n)}_{(3n+4)k+i+K}} \right)  
\end{align*}
for $0 \leq i \leq 3n+3$.  The purpose of this section is to prove the following:

\begin{thm} \label{th:lim-xn}
	The sequence $\{ \rep(x^{(n)}) \}_n$ converges to 
	$r_1 = (48+\sqrt{10})/31$ as $n \to \infty$. 
\end{thm}

The proof of Theorem \ref{th:lim-xn} will be accomplished by the following three steps: 

\begin{itemize}
\item[\textbf{Step 1}] we prove
\begin{gather*}
\rep(x^{(n)})=\displaystyle \min_{0 \leq i \leq 3n+3} \left\{ \zeta^{(n)}_{i}, \xi^{(n)}_{ i} \right\} 
\end{gather*}
for all $n \ge 1$;
\item[\textbf{Step 2}]  we prove
\begin{gather*}
\lim_{n \to \infty} \zeta^{(n)}_{0} = r_1 ;
\end{gather*} 
\item[\textbf{Step 3}] we prove
\begin{gather*}
 \zeta^{(n)}_{0} = \displaystyle \min_{0 \leq i \leq 3n+3} \left\{ \zeta^{(n)}_{i}, \xi^{(n)}_{ i} \right\} 
\end{gather*}
holds  for all sufficiently large $n$. 
\end{itemize}
We need a series of lemmas in both Step 2 and Step 3. 
To accomplish Step 3, we classify $0 \le i \le 3n+3$ by modulo $3$ and describe both $\zeta^{(n)}_{3m+j}$ and $\xi^{(n)}_{3m+j}$ in terms of  continued fractions $e^{(n)}_j(m)$ introduced in \S 3,  and then we estimate both $\lim_{n \to \infty} \zeta^{(n)}_{3m+j}$ and $\lim_{n \to \infty} \xi^{(n)}_{3m+j}$ (Lemmas \ref{lem:0case} and \ref{lem:12case}). These lemmas and some simple result on double sequences (Lemma \ref{lem:memo2}) give the equation of Step 3. The case $m = j = 0$, i.e., $\zeta^{(n)}_0$ is previously calculated in Step 2. 

In the following,  we often omit the superscript $(n)$ to simplify the notations if no confusion arises. For example, we write simply $W_j, q_j, \eta_j, t_j, \cdots $ for $W^{(n)}_j, q^{(n)}_j, \eta^{(n)}_j, t^{(n)}_j, \cdots$.

\begin{prop} \label{lem:rep-min}
	For $n \ge 1$, 
	\begin{gather*}
		\rep(x^{(n)})=\displaystyle \min_{0 \leq i \leq 3n+3} \left\{ \zeta^{(n)}_{i}, \xi^{(n)}_{ i} \right\} . 
	\end{gather*}
\end{prop}

\begin{proof}
Let $x = x^{(n)}$ and $k$ be an integer with $k>K$. 
For $0 \leq i \leq 3n+3$, 
we have 
\begin{align*}
	& \dfrac{r(|W_{(3n+4)k+i+K}|+q_{(3n+4)k+i+K}+q_{(3n+4)k+i-1+K}-2, x)}
			{|W_{(3n+4)k+i+K}|+q_{(3n+4)k+i+K}+q_{(3n+4)k+i-1+K}-2} \\
		&\qquad <1+\dfrac{1+\eta_{(3n+4)k+i+K}}{t_{(3n+4)k+i+K}+1+\eta_{(3n+4)k+i+K}}+\varepsilon_{(3n+4)k+i+K}
\end{align*}
by lemma \ref{buglem:7.6}(1) and 
\begin{gather*}
	\dfrac{r(q_{(3n+4)k+i+K}+q_{(3n+4)k+i-1+K}-2, x)}{q_{(3n+4)k+i+K}+q_{(3n+4)k+i-1+K}-2}
			<1+\dfrac{t_{(3n+4)k+i+K}+\eta_{(3n+4)k+i+K}}{1+\eta_{(3n+4)k+i+K}}+\varepsilon_{(3n+4)k+i+K}
\end{gather*}
by lemma \ref{buglem:7.6}(2). 
These show
\begin{align*}
	\rep(x)
		&\leq \displaystyle \min_{0 \leq i \leq 3n+3} \left\{ \zeta^{(n)}_{i} , \xi^{(n)}_{i} \right\}. 
\end{align*}
Similarly as in \cite[Proof of Theorem 3.4]{bugeaud}, 
$r(m, x)$ satisfies the following inequality:
\begin{align*}
	r(m, x)
	\geq \begin{cases}
		m+q_k+q_{k-1} & (q_k+q_{k-1}-1 \leq m \leq |W_k|+q_k+q_{k-1}-2) ,  \\
		m+|W_k|+q_k+q_{k-1} & (|W_k|+q_k+q_{k-1}-1 \leq m \leq q_{k+1}+q_k-2) . 
	\end{cases}
\end{align*}
This implies that 
\begin{align*}
	\rep(x) 
		&\geq \begin{cases}
			\displaystyle \liminf_{k \to \infty} \dfrac{|W_k|+2q_k+2q_{k-1}-2}{|W_k|+q_k+q_{k-1}-2} \\
			\displaystyle \liminf_{k \to \infty} \dfrac{|W_k|+q_{k+1}+2q_k+q_{k-1}-2}{q_{k+1}+q_k-2}
			\end{cases} \\
		&= \begin{cases}
			1+\displaystyle \liminf_{k \to \infty} \dfrac{1+\eta_k}{t_k+1+\eta_k} \\
			1+\displaystyle \liminf_{k \to \infty} \dfrac{t_{k}+\eta_{k}}{1+\eta_{k}}.
			\end{cases}
\end{align*}
Hence we obtain 
\begin{align*}
	\rep(x)
		&\geq \displaystyle \min_{0 \leq i \leq 3n+3} \left\{ \zeta^{(n)}_{i}, \xi^{(n)}_{i} \right\}. 
\end{align*}
This completes the proof. 
\end{proof}

We proceed to Step 2. 

\begin{prop} \label{prop:appendix}
	The sequence $\{ \zeta^{(n)}_{0} \}_n$ converges to $r_1$ as $n \to \infty$. 
\end{prop}

Since 
\begin{gather*}
\zeta^{(n)}_{0} =\liminf_{k \to \infty} 
			\left( 1+\dfrac{1+\eta^{(n)}_{(3n+4)k+K}}{t^{(n)}_{(3n+4)k+K}+1+\eta^{(n)}_{(3n+4)k+K}} \right) , 
\end{gather*}
we have to compute the limits of 
both $\eta_{(3n+4)k+K} = \eta^{(n)}_{(3n+4)k+K}$ and $t_{(3n+4)k+K}=t^{(n)}_{(3n+4)k+K}$ as $k \to \infty$ to prove this Proposition.  Our computation method is essentially same as in Proof of Lemma \ref{lem:r2}. 
By the same argument as in \cite[Proof of Theorem 3.4]{bugeaud}, 
we have
\begin{align}
	\begin{aligned} \label{eq:eta}
	e^{(n)}_0(m) & = \displaystyle \lim_{k \to \infty} \eta_{(3n+4)k+3m+K} , \qquad (m=0,1,2,\cdots n+1) \\ 
	 e^{(n)}_1(m) & = \displaystyle \lim_{k \to \infty} \eta_{(3n+4)k+3m+1+K} , \qquad  (m=0,1,2,\cdots n) \\
	e^{(n)}_2(m) & = \displaystyle \lim_{k \to \infty} \eta_{(3n+4)k+3m+2+K},  \qquad (m=0,1,2,\cdots n). 
  \end{aligned}
\end{align}
To compute the limit of $t_{(3n+4)k+K}$, let
\begin{align}
	\begin{aligned} \label{def:beta}
	s^{(n)}_k=&  \sum_{\ell = 0}^{(3n+4)(k+1) - 1} \prod_{j = \ell}^{(3n+4)(k+1)} \eta_{j+K}, \\
	\alpha^{(n)}_k=& \sum_{\ell = 0}^{3n+3} \prod_{j = \ell}^{3n+4} \eta_{(3n+4)k + j + K}, \qquad 	\beta^{(n)}_k= \prod_{j = 1}^{3n+4} \eta_{(3n+4)k + j + K} . 
	\end{aligned}
\end{align}
Then 
\begin{gather*}
		s^{(n)}_k=\alpha^{(n)}_k+\beta^{(n)}_k s^{(n)}_{k-1} 
\end{gather*}
and $t_{(3n+4)(k+1)+K}$ is represented as 
\begin{gather*}
t_{(3n+4)(k+1)+K}=s^{(n)}_k+\dfrac{|W_K|}{q_{(3n+4)(k+1)+K}}.
\end{gather*}
Since $|W_K|/{q_{(3n+4)(k+1)+K}} \to 0$ as $k \to \infty$,  we will show the existence of the limit of $s^{(n)}_k$. 

\begin{lem} \label{lem:limbeta}
	$\displaystyle \lim_{n \to \infty} \displaystyle \lim_{k \to \infty} \beta^{(n)}_k=0$. 
\end{lem}

\begin{proof}
We use  (\ref{eq:eta}). Note that the limit of $\eta_{(3n+4)(k+1) + K}$ as $k \to \infty$ equals  $e^{(n)}_0(0)$, but not $e^{(n)}_1(n+1)$. Thus, the limit of $\beta^{(n)}_k$ as $k \to \infty$ is described as
\begin{align*}
	\displaystyle \lim_{k \to \infty} \beta^{(n)}_k
		&=e^{(n)}_0(0)  \prod_{j=0}^n \left \{ e^{(n)}_0(j+1)e_2^{(n)}(j)e^{(n)}_1(j) \right \} \\
		&= e^{(n)}_0(n+1) \prod_{j=0}^n \left \{ e^{(n)}_2(j) e^{(n)}_1(j)e^{(n)}_0(j)  \right \} . 
\end{align*}
It follows from Lemma \ref{lem:e-property}(1) that
\begin{gather} \label{eq:tau}
	e^{(n)}_2(m) e^{(n)}_1(m) e^{(n)}_0(m) < \dfrac{3}{4} \cdot \dfrac{2}{5} \cdot \dfrac{3}{5}=\dfrac{9}{50}.  
\end{gather}
This implies
\begin{gather*} 
	\displaystyle \lim_{k \to \infty} \beta^{(n)}_k
				<\dfrac{3}{5} \cdot \left( \dfrac{9}{50} \right)^{n+1}, 
\end{gather*}
whence $\displaystyle \lim_{n \to \infty} \displaystyle \lim_{k \to \infty} \beta^{(n)}_k=0$. 
\end{proof}

Let 
\begin{align*}
	\sigma^{(n)}(m) &= 1+e^{(n)}_2(m)+e^{(n)}_2(m) e^{(n)}_1(m) , \\
	\tau^{(n)}(m) &=e^{(n)}_2(m) e^{(n)}_1(m) e^{(n)}_0(m). 
\end{align*} 
We note that $\sigma^{(n)}(m)$ is the constant $2$ for all $m \ge 0$ by (\ref{eq:e-relation}). 
To describe the limit of $\alpha^{(n)}_k$ as $k \to \infty$ , define $\gamma^{(n)}(m)$ recursively by
\begin{align} \label{eq:gamma}
		\gamma^{(n)}(m)
		&= \sigma^{(n)}(m) + \tau^{(n)}(m) \gamma^{(n)}(m-1)   =  2+  \tau^{(n)}(m) \gamma^{(n)}(m-1) \qquad (m \ge 1), \\
		\gamma^{(n)}(0)&=1+e^{(n)}_2(0) \left \{ 1+e^{(n)}_1(0)(1+e^{(n)}_0(0))\right \} = 2 + \tau^{(n)}(0). \nonumber 
\end{align}
Then the limit of $\alpha^{(n)}_k$ as $k \to \infty$ is represented as    
\begin{gather*}
	\displaystyle \lim_{k \to \infty} \alpha^{(n)}_k = e^{(n)}_0(0) e^{(n)}_0(n+1) \gamma^{(n)}(n).
\end{gather*} 
 By  (\ref{eq:e-gen}), both $\tau^{(n)}(m)$ and $\gamma^{(n)}(m)$ are defined 
for all $n \geq 1$ and all $m \geq 0$.

\begin{lem} \label{lem:gamma}
	The double sequence $\left\{ \gamma^{(n)}(m) \right\}_{n, m}$ has the following properties:
	
		$(1)$ $\left\{ \gamma^{(n)}(m) \right\}_{n, m}$ is bounded. 
		
		$(2)$ There exists the limit $\gamma(m)=\displaystyle \lim_{n \to \infty} \gamma^{(n)}(m)$
			for every $m$. 
		
		$(3)$ $\gamma^{(n)}(m)$ converges uniformly to $\gamma(m)$ as $n \to \infty$. 
		
		$(4)$ There exists the limit $\gamma=\displaystyle \lim_{m \to \infty} \gamma(m)$. 
		
		$(5)$ $\gamma^{(n)}(n)$ converges to $\gamma$ as $n \to \infty$. 
\end{lem}

\begin{proof}
(1)\; Since  $0<e^{(n)}_i(m)<3/4$ 
	for all $n \geq 1, m \geq 0$ and $i = 0,1,2$ by Lemma \ref{lem:e-property}(1), $\gamma^{(n)}(m)$ is estimated as 
	\begin{gather*}
		\gamma^{(n)}(m) 
			<1+\left( \dfrac{3}{4} \right)+\left( \dfrac{3}{4} \right)^2+\left( \dfrac{3}{4} \right)^3+
				\cdots+\left( \dfrac{3}{4} \right)^{3m+3} < 4
	\end{gather*}
	
(2)\; We prove this by induction. By Lemma \ref{lem:e-property}(3), 
	there exist 
	\begin{gather*}
	\tau(m)=\displaystyle \lim_{n \to \infty} \tau^{(n)}(m)
	\end{gather*}
	 for every $m \geq 0$, and hence 
	\begin{gather*}
		\lim_{n \to \infty} \gamma^{(n)}(0)= 2 + \tau(0).
	\end{gather*}
	Assume that $\gamma^{(n)}(m-1)$ converges to 
	$\gamma(m-1)$ as $n \to \infty$ for $m \ge 1$.  Then $\gamma^{(n)}(m)=2+\tau^{(n)}(m) \gamma^{(n)}(m-1)$ 
	converges to $2+\tau(m) \gamma(m-1)$ as $n \to \infty$. 

(3)\; We estimate  
	\begin{gather*}
		\gamma(m)-\gamma^{(n)}(m)
			= \left( \tau(m)-\tau^{(n)}(m) \right) \gamma(m-1) 
			 +\tau^{(n)}(m) \left( \gamma(m-1)-\gamma^{(n)}(m-1) \right) . 
	\end{gather*}
	Since $\tau^{(n)}(m)$  converges uniformly to $\tau(m)$ as $n \to \infty$  by Lemma \ref{lem:e-property}(3), 
	for an arbitrary $\varepsilon > 0$, there exists $N$ such that 
	\begin{gather*}
		\left| \tau(m)-\tau^{(n)}(m) \right| < \varepsilon 
	\end{gather*}
	holds for all $n \geq N$ and all $m \geq 0$.  
	Since $\left| \gamma(m-1) \right|\leq  4$ by (1) 
	and $\left| \tau^{(n)}(m) \right|<9/50$ by (\ref{eq:tau}), 
	we have, for all $n \geq N$ and $m \geq 0$, 
	\begin{align*}
		\left| \gamma(m)-\gamma^{(n)}(m) \right|
			&< 4\varepsilon 
				+\dfrac{9}{50} \left| \gamma(m-1)-\gamma^{(n)}(m-1) \right| \\
			&<4\varepsilon +\dfrac{9}{50} 
				\left( 4\varepsilon +\dfrac{9}{50} \left| \gamma(m-2)-\gamma^{(n)}(m-2) \right| \right) \\
			&<4\varepsilon \left( 1+\dfrac{9}{50}+\left( \dfrac{9}{50} \right)^2+\cdots
				+ \left( \dfrac{9}{50} \right)^{m-1} \right) 
				+\left( \dfrac{9}{50} \right)^m \left| \gamma(0)-\gamma^{(n)}(0) \right| \\
			& < 4\varepsilon \cdot \dfrac{50}{41}+\left| \gamma(0)-\gamma^{(n)}(0) \right|. 
	\end{align*}
	This implies that $\lim_{n \to \infty} \gamma^{(n)}(m)=\gamma(m)$ is  uniformly 
	convergent with respect to $m$. 

(4)\; Consider the recurrence $\gamma(m)=2+\tau(m) \gamma(m-1)$. 
	By Lemma \ref{lem:e-property} (4),  we have
	\begin{gather*}
	 \tau= \lim_{m \to \infty} \tau(m) = \displaystyle \lim_{m \to \infty} e_2(m) e_1(m) e_0(m)=-3+\sqrt{10} . 
    \end{gather*}
		Since the sequence $\left\{ \gamma^{(n)}(m) \right\}_{n, m }$ is bounded by (1), 
	both 
	\begin{gather*}
	\gamma^{-}=\displaystyle \limsup_{m \to \infty} \gamma(m) \quad  \text{and} \quad 
	\gamma_{-}=\displaystyle \liminf_{m \to \infty} \gamma(m)
	\end{gather*}
	are finite positive numbers. 
	Then we have 
	\begin{align*}
		\gamma^{-} & =\displaystyle \limsup_{m \to \infty} \left( 2+\tau(m) \gamma(m-1) \right)
			\leq 2+\tau \gamma^{-} \\
		\gamma_{-} &=\displaystyle \liminf_{m \to \infty} \left( 2+\tau(m) \gamma(m-1) \right)
			\geq 2+\tau \gamma_{-} . 
	\end{align*}
	This implies 
	\begin{gather*}
		\dfrac{2}{1-\tau} \leq \gamma_{-} \leq \gamma^{-} \leq \dfrac{2}{1-\tau}. 
	\end{gather*}
	Therefore $\gamma(m)$ converges to $\gamma = 2/(1 - \tau)$ as $m \to \infty$. 
	
(5)\; By (3) and (4), the estimate
	\begin{gather*}
		\left| \gamma-\gamma^{(n)}(n) \right| \leq 
				\left| \gamma-\gamma(n) \right|+\left| \gamma(n)-\gamma^{(n)}(n) \right|
	\end{gather*}
	gives  $\displaystyle \lim_{n \to \infty} \gamma^{(n)}(n)=\gamma$. 
\end{proof}

\begin{lem} \label{lem:limalpha}
	$\displaystyle \lim_{n \to \infty} \displaystyle \lim_{k \to \infty} \alpha^{(n)}_k=(10+\sqrt{10})/15$. 
\end{lem}

\begin{proof}
By the proof of Lemma \ref{lem:gamma} (4), 
\begin{gather*}
	\displaystyle \lim_{n \to \infty} \gamma^{(n)}(n)= \dfrac{2}{1-\tau}
	=\dfrac{4+\sqrt{10}}{3} . 
\end{gather*}
Combining this with Lemma \ref{lem:e-property} (2), we obtain
\begin{gather*}
	\displaystyle \lim_{n \to \infty} \displaystyle \lim_{k \to \infty} \alpha^{(n)}_k 
		=\displaystyle \lim_{n \to \infty} e^{(n)}_0(0) e^{(n)}_0(n+1) \gamma^{(n)}(n) 
		=\dfrac{10+\sqrt{10}}{15} . 
\end{gather*}
\end{proof}

\begin{proof}[Proof of Proposition \ref{prop:appendix}]
Since $\eta^{(n)}_j \le 3/4$ for all $j  > K+1$ by the definition of $\theta^{(n)}$,  the sequence $\left\{ s^{(n)}_k \right\}_k$ is bounded, so 
both $\displaystyle \limsup_{k \to \infty} s^{(n)}_k$ and 
	$\displaystyle \liminf_{k \to \infty} s^{(n)}_k$ are finite positive numbers. 
By $s^{(n)}_k = \alpha^{(n)}_k + \beta^{(n)}_k s^{(n)}_{k-1}$ and the same argument as in the proof of Lemma \ref{lem:gamma} (4), we obtain  
\begin{gather*}
		\displaystyle \limsup_{k \to \infty} s^{(n)}_k  
		=\displaystyle \liminf_{k \to \infty} s^{(n)}_k
		=\dfrac{ \displaystyle \lim_{k \to \infty} \alpha^{(n)}_k}
			{1- \displaystyle \lim_{k \to \infty} \beta^{(n)}_k}.
\end{gather*}
From Lemmas \ref{lem:limbeta} and \ref{lem:limalpha}, it follows 
\begin{gather*}
	\displaystyle \lim_{n \to \infty} \displaystyle \lim_{k \to \infty} t_{(3n+4)(k+1)+K}
		= \displaystyle \lim_{n \to \infty} \displaystyle \lim_{k \to \infty} s^{(n)}_k
		=\dfrac{10+\sqrt{10}}{15} . 
\end{gather*}
By the definition of $\zeta^{(n)}_0$, we conclude
\begin{gather*}
	\displaystyle \lim_{n \to \infty} \zeta^{(n)}_{0}
		=1+\dfrac{1+e_0}{\frac{10+\sqrt{10}}{15}+1+e_0} 
		=\dfrac{48+\sqrt{10}}{31} . 
\end{gather*}
\end{proof}

Step 2 is completed. We proceed to Step 3. 

\begin{lem} \label{lem:memo2}
	Let $\left\{ \psi^{(n)}_0(m) \right\}_{n, m}, \left\{ \psi^{(n)}_1(m) \right\}_{n, m}, \cdots, 
	\left\{ \psi^{(n)}_k(m) \right\}_{n, m}$ be a family of $(k+1)$ double sequences of real numbers. 
	Let 
	\begin{gather*}
		r_n=\displaystyle \min_{0 \leq m \leq n} 
				\left\{ \psi^{(n)}_0(m), \psi^{(n)}_1(m), \cdots, \psi^{(n)}_k(m) \right\} . 
	\end{gather*}
	Assume that $\left\{ \psi^{(n)}_0(m) \right\}_{n, m}, \cdots, \left\{ \psi^{(n)}_k(m) \right\}_{n, m}$
	are satisfying
	\begin{enumerate}
		\item[$(\text{\romannumeral 1)}$] for every $j=0, 1, \cdots, k$ and all $m \geq 0$, 
			there exists $\displaystyle \lim_{n \to \infty} \psi^{(n)}_j(m)=\psi_j(m)$ 
			and this convergence is uniform with respect to $m$, 
		\item[$\text{(\romannumeral 2)}$] there exists a constant $\delta>0$ such that
			both $\psi_0(0)+\delta \leq \psi_0(m+1)$ and $\psi_0(0)+\delta \leq \psi_j(m)$ hold 
			for all $j=1, 2, \cdots, k$ and all $m \geq 0$. 
	\end{enumerate}
	Then we have $r_n=\psi^{(n)}_0(0)$ for all sufficiently large $n$. 
	In particular, $ r_n$ converges to $\psi_0(0)$ as $n \to \infty$. 
\end{lem}

\begin{proof}
By the assumption (\romannumeral 1), 
there exists a sufficiently large $N > 0$ such that 
\begin{gather*}
	\left|\psi_j^{(n)}(m) - \psi_j(m) \right| < \delta/2
\end{gather*}
holds for all $n \geq N$, $m \geq 0$ and $j = 0, \cdots, k$. 
Then, by the assumption (\romannumeral 2), we have
\begin{gather*}
	\psi_0^{(n)}(0) < \psi_0(0) + \frac{\delta}{2}
		\leq \psi_0(m+1) - \frac{\delta}{2}
		< \psi_0^{(n)}(m+1)
\end{gather*}
and
\begin{gather*}
	\psi_0^{(n)}(0) < \psi_0(0) + \frac{\delta}{2}
		\leq \psi_j(m) - \frac{\delta}{2}
		< \psi_j^{(n)}(m)
\end{gather*}
for all $j = 1, \cdots, k$ and $m \geq 0$. 
Therefore, $r_n$ is equal to $\psi_0^{(n)}(0)$ when $n \geq N$.
\end{proof}

Recall $\zeta^{(n)}_{i}$ and $\xi^{(n)}_{i}$:
\begin{align*}
	\zeta^{(n)}_{i}&=\liminf_{k \to \infty} 
			\left( 1+\dfrac{1+\eta^{(n)}_{(3n+4)k+i+K}}{t^{(n)}_{(3n+4)k+i+K}+1+\eta^{(n)}_{(3n+4)k+i+K}} \right) , \\
	\xi^{(n)}_{i}&=\liminf_{k \to \infty} 
			\left( 1+\dfrac{t^{(n)}_{(3n+4)k+i+K}+\eta^{(n)}_{(3n+4)k+i+K}}{1+\eta^{(n)}_{(3n+4)k+i+K}} \right) . 
\end{align*}
To compute $\zeta^{(n)}_{i}$ and $\xi^{(n)}_{i}$, we partition $i$ into 3 classes by modulo 3. 
First we consider the case $i = 3m$. Define $s^{(n)}_{0,k}(m)$, $\alpha^{(n)}_{0,k}(m)$ and $\beta^{(n)}_{0,k}(m)$ similarly as $s^{(n)}_k$, $\alpha^{(n)}_k$ and $\beta^{(n)}_k$ in  (\ref{def:beta}), namely 
\begin{align}
	\begin{aligned} \label{def:0mod3}
	s^{(n)}_{0,k}(m)=&  \sum_{\ell = 0}^{(3n+4)(k+1)+3m - 1} \prod_{j = \ell}^{(3n+4)(k+1) + 3m} \eta^{(n)}_{j+K}, \\
	\alpha^{(n)}_{0,k}(m)=& \sum_{\ell = 0}^{3n+3} \prod_{j = \ell}^{3n+4 } \eta^{(n)}_{(3n+4)k  + 3m + j + K}, \qquad 	\beta^{(n)}_{0,k}(m)= \prod_{j = 1}^{3n+4} \eta^{(n)}_{(3n+4)k + 3m +  j + K} . 
	\end{aligned}
\end{align}
Then the limit of $t^{(n)}_{(3n+4)k+3m+K}$ as $k \to \infty$ is represented as  
\begin{gather*}
	\displaystyle \lim_{k \to \infty} t^{(n)}_{(3n+4)k+3m+K} 
	= \frac{\displaystyle\lim_{k \to \infty} \alpha^{(n)}_{0,k}(m)}
		{1 - \displaystyle\lim_{k \to \infty} \beta^{(n)}_{0,k}(m)} 
\end{gather*}
by the same argument as in Proof of Proposition \ref{prop:appendix}. We describe
\begin{gather*}
 \lim_{k \to \infty} \beta^{(n)}_{0,k}(m) = \prod_{j=1}^{3n+4} \lim_{k \to \infty} \eta^{(n)}_{(3n+4)k + 3m +  j + K} 
\end{gather*}
by using (\ref{eq:eta}). Here we note that 
\begin{gather*}
\lim_{k \to \infty} \eta^{(n)}_{(3n+4)k + 3m + 3(n-m+4) + K} = e^{(n)}_0(0) \ne e^{(n)}_1(n+1) , \\
\lim_{k \to \infty} \eta^{(n)}_{(3n+4)k + 3m + 3(n-m+4) + 1 + K} = e^{(n)}_1(0) \ne e^{(n)}_2(n+1) ,
\end{gather*}
and so on. 
Let 
\begin{gather*}
\tau^{(n)}_0(m) = e^{(n)}_0(m) e^{(n)}_2(m-1) e^{(n)}_1(m-1)
\end{gather*}
and 
\begin{gather*}
\sigma^{(n)}_0(m) = e^{(n)}_0(m) e^{(n)}_2(m-1) + \tau^{(n)}_0(m) + \tau^{(n)}_0(m) e^{(n)}_0(m-1).
\end{gather*}
Then we obtain 
\begin{gather*}
 \nu^{(n)}_0 := \lim_{k \to \infty} \beta^{(n)}_{0,k}(m) =  \tau^{(n)}_0(n+1) \tau^{(n)}_0(n) \tau^{(n)}_0(n-1)\cdots \tau^{(n)}_0(1) e^{(n)}_0(0),
\end{gather*}
which is not depend on $m$. 
To describe $\displaystyle \lim_{k \to \infty} \alpha^{(n)}_{0,k}(m)$, we need more notations. For a finite sequence $\{ c_j\}_{j=1}^N$ of real numbers, we set
\begin{gather*}
\Lambda(\{c_j\}) = \sum_{\ell =2}^N \prod_{j = 1}^\ell c_j  = c_1c_2 + c_1c_2c_3 + \cdots + c_1c_2 \cdots c_N. 
\end{gather*}
Let $E^{n,m}_0$ be the sequence of length $3n+5$ whose first $3m+1$ terms are 
\begin{gather*} e^{(n)}_0(m), e^{(n)}_2(m-1), e^{(n)}_1(m-1), e^{(n)}_0(m-1), e^{(n)}_2(m-2), \cdots, e^{(n)}_0(1), e^{(n)}_2(0), e^{(n)}_1(0), e^{(n)}_0(0) 
\end{gather*}
and remaining terms are 
\begin{gather*}
e^{(n)}_0(n+1), e^{(n)}_2(n), e^{(n)}_1(n), e^{(n)}_0(n), e^{(n)}_2(n-1), \cdots, e^{(n)}_0(m+1), e^{(n)}_2(m), e^{(n)}_1(m), e^{(n)}_0(m) .
\end{gather*}
Then it follows from the definition of  $\alpha^{(n)}_{0,k}(m)$ that 
\begin{gather*}
\lambda^{(n)}_0(m) := \lim_{k \to \infty} \alpha^{(n)}_{0,k}(m) = \Lambda(E^{n,m}_0) .
\end{gather*}
We note  that $\lambda^{(n)}_0(0) = e^{(n)}_0(0)e^{(n)}_0(n+1)\gamma^{(n)}(n)$. 
Let $\chi^{(n)}_0(m-1)$ be the last three terms of $\lambda^{(n)}_0(m-1)$, namely 
\begin{align*}	\chi^{(n)}_0(m-1)
		&=e^{(n)}_0(m-1) e^{(n)}_2(m-2) \cdots
			e^{(n)}_0(0) e^{(n)}_0(n+1) e^{(n)}_2(n) \cdots 
			e^{(n)}_0(m) e^{(n)}_2(m-1) \\
		&\quad +e^{(n)}_0(m-1) e^{(n)}_2(m-2) \cdots
			e^{(n)}_0(0) e^{(n)}_0(n+1) e^{(n)}_2(n) \cdots 
			e^{(n)}_2(m-1) e^{(n)}_1(m-1) \\
		&\quad +e^{(n)}_0(m-1) e^{(n)}_2(m-2) \cdots
			e^{(n)}_0(0) e^{(n)}_0(n+1) e^{(n)}_2(n) \cdots 
			e^{(n)}_1(m-1) e^{(n)}_0(m-1) 	 .
\end{align*}
Then $\lambda^{(n)}_0(m)$ satisfies  
\begin{gather}\label{eq:lambda0-recurrence}
	\lambda^{(n)}_0(m) = \sigma^{(n)}_0(m)+\tau^{(n)}_0(m)  \left ( \lambda^{(n)}_0(m-1)
		- \chi^{(n)}_0(m-1) \right ) .  
\end{gather}
By using these notations, both $\zeta^{(n)}_{3m}$ and $\xi^{(n)}_{3m}$ are described by
\begin{gather*}
	\zeta^{(n)}_{3m}
			= 1+\dfrac{1+e^{(n)}_0(m)}{\frac{\lambda^{(n)}_0(m)}{1 - \nu^{(n)}_0}+1+e^{(n)}_0(m)} , \qquad 
	\xi^{(n)}_{3m}
		= 1+\dfrac{\frac{\lambda^{(n)}_0(m)}{1 - \nu^{(n)}_0}+e^{(n)}_0(m)}{1+e^{(n)}_0(m)}
\end{gather*}
for $m = 0, 1, \cdots, n+1$. 

By a similar fashion ,we can describe both $\zeta^{(n)}_i$ and $\xi^{(n)}_i$ for  $i = 3m+1 , 3m+2$.  We state only the conclusion. Let $E^{n,m}_1$ be the sequence of length $3n+5$ whose first $3m+2$ terms are 
\begin{gather*} e^{(n)}_1(m), e^{(n)}_0(m), e^{(n)}_2(m-1), e^{(n)}_1(m-1), e^{(n)}_0(m-1), \cdots, e^{(n)}_0(1), e^{(n)}_2(0), e^{(n)}_1(0), e^{(n)}_0(0) 
\end{gather*}
and remaining terms are 
\begin{gather*}
e^{(n)}_0(n+1), e^{(n)}_2(n), e^{(n)}_1(n), e^{(n)}_0(n), e^{(n)}_2(n-1), \cdots, e^{(n)}_0(m+1), e^{(n)}_2(m), e^{(n)}_1(m) .
\end{gather*}
Let
\begin{align*}
	\tau^{(n)}_1(m) &= e^{(n)}_1(m) e^{(n)}_0(m) e^{(n)}_2(m-1) , \\
	\sigma^{(n)}_1(m) &= e^{(n)}_1(m) e^{(n)}_0(m)
			+\tau^{(n)}_1(m) + \tau^{(n)}_1(m) e^{(n)}_1(m) \\
	\nu^{(n)}_1 & = \tau^{(n)}_1(n)\tau^{(n)}_1(n-1) \cdots \tau^{(n)}_1(1)
		e^{(n)}_1(0)e^{(n)}_0(0)e^{(n)}_0(n+1)e^{(n)}_2(n) \\
	\chi^{(n)}_1(m-1)
		&=e^{(n)}_1(m) e^{(n)}_0(m) \cdots
			e^{(n)}_0(0) e^{(n)}_0(n+1) e^{(n)}_2(n) \cdots 
			e^{(n)}_1(m) e^{(n)}_0(m) \\
		&\quad +e^{(n)}_1(m) e^{(n)}_0(m)  \cdots
			e^{(n)}_0(0) e^{(n)}_0(n+1) e^{(n)}_2(n) \cdots 
			e^{(n)}_1(m) e^{(n)}_0(m) e^{(n)}_2(m-1) \\
		&\quad +e^{(n)}_1(m) e^{(n)}_0(m)  \cdots
			e^{(n)}_0(0) e^{(n)}_0(n+1) e^{(n)}_2(n) \cdots 
			e^{(n)}_0(m) e^{(n)}_2(m-1) e^{(n)}_1(m-1) .
\end{align*}
Similarly, let $E^{n,m}_2$ be the sequence of length $3n+5$ whose first $3m + 3$ terms are 
\begin{gather*} e^{(n)}_2(m), e^{(n)}_1(m), e^{(n)}_0(m), e^{(n)}_2(m-1), e^{(n)}_1(m-1), \cdots, e^{(n)}_0(1), e^{(n)}_2(0), e^{(n)}_1(0), e^{(n)}_0(0) 
\end{gather*}
and remaining terms are 
\begin{gather*}
e^{(n)}_0(n+1), e^{(n)}_2(n), e^{(n)}_1(n), e^{(n)}_0(n), e^{(n)}_2(n-1), \cdots, e^{(n)}_0(m+1), e^{(n)}_2(m) .
\end{gather*}
Let
\begin{align*}
    \tau^{(n)}_2(m) &= e^{(n)}_2(m) e^{(n)}_1(m) e^{(n)}_0(m) , \\
	\sigma^{(n)}_2(m) &= e^{(n)}_2(m) e^{(n)}_1(m)
			+ \tau^{(n)}_2(m) + \tau^{(n)}_2(m) e^{(n)}_2(m-1) , \\
	\nu^{(n)}_2 & = \tau^{(n)}_2(n)\tau^{(n)}_2(n-1) \cdots \tau^{(n)}_2(0)
			e^{(n)}_0(n+1) \\
	\chi^{(n)}_2(m-1)
		&=e^{(n)}_2(m) e^{(n)}_1(m) \cdots
			e^{(n)}_0(0) e^{(n)}_0(n+1) e^{(n)}_2(n) \cdots 
			e^{(n)}_2(m) e^{(n)}_1(m) \\
		&\quad +e^{(n)}_2(m) e^{(n)}_1(m)  \cdots
			e^{(n)}_0(0) e^{(n)}_0(n+1) e^{(n)}_2(n) \cdots 
			e^{(n)}_2(m) e^{(n)}_1(m) e^{(n)}_0(m) \\
		&\quad +e^{(n)}_2(m) e^{(n)}_1(m)  \cdots
			e^{(n)}_0(0) e^{(n)}_0(n+1) e^{(n)}_2(n) \cdots 
			e^{(n)}_1(m) e^{(n)}_0(m) e^{(n)}_2(m-1)  .
\end{align*}
Then we have
\begin{gather*}		
	\zeta^{(n)}_{3m+j}  =  1+\dfrac{1+e^{(n)}_j(m)}{\frac{\lambda^{(n)}_j(m)}{1 - \nu^{(n)}_j}+1+e^{(n)}_j(m)},  \qquad 
	\xi^{(n)}_{3m+ j} 
		= 1+\dfrac{\frac{\lambda^{(n)}_j(m)}{1 - \nu^{(n)}_j}+e^{(n)}_j(m)}{1+e^{(n)}_j(m)}		 
\end{gather*}
for $j = 1,2$ and $m = 0, 1, \cdots n$, where $\lambda^{(n)}_j(m)$ denotes $\Lambda(E^{n,m}_j)$.  This $\lambda^{(n)}_j(m)$ satisfies
\begin{gather}\label{eq:lambda12-recurrence}
	\lambda^{(n)}_j(m) = \sigma^{(n)}_j(m)+\tau^{(n)}_j(m)  \left (\lambda^{(n)}_j(m-1)
		- \chi^{(n)}_j(m-1) \right ).
\end{gather}

To apply Lemma \ref{lem:memo2}, we put 
\begin{gather*}
	\varphi^{(n)}_j(m) = \zeta^{(n)}_{3m+j}, \qquad 	\Phi^{(n)}_j(m) = \xi^{(n)}_{3m+j}
\end{gather*}
for $j = 0,1,2$. 

\begin{lem} \label{lem:asm(1)}
	For $j=0,1, 2$ and every $m \geq 0$, 
	there exists 
	\begin{gather*}
		\lim_{n \to \infty} \varphi^{(n)}_j(m)=\varphi_j(m),\quad 
		\lim_{n \to \infty} \Phi^{(n)}_j(m)=\Phi_j(m) , 
	\end{gather*}
	and these convergences are uniform with respect to $m$. 
\end{lem}

\begin{proof}
We prove the case $j = 0$. 
By Lemma \ref{lem:e-property}(3), 
\begin{gather*}
	e_0(m)=\displaystyle \lim_{n \to \infty} e^{(n)}_0(m), \quad 
	e_1(m)=\displaystyle \lim_{n \to \infty} e^{(n)}_1(m), \quad  
	e_2(m)=\displaystyle \lim_{n \to \infty} e^{(n)}_2(m)
\end{gather*}
are uniformly convergent with respect to $m$. 
Hence, both 
\begin{gather*}
	\sigma_0(m)=\displaystyle \lim_{n \to \infty} \sigma^{(n)}_0(m) 
	\qquad \text{and} \qquad 
	\tau_0(m)=\displaystyle \lim_{n \to \infty} \tau^{(n)}_0(m) 
\end{gather*}
are also uniformly convergent with respect to $m$. 
Since $e^{(n)}_i(m) < 3/4$ for all $i,n,m$, we observe that 
\begin{gather*}
	|\tau^{(n)}_0(m)| < \left ( \frac{3}{4} \right )^3 \,. 
\end{gather*}
Since $\chi^{(n)}_0(m-1)$ is written as 
\begin{gather*}
	\chi^{(n)}_0(m-1) 
		= e^{(n)}_0(0) \left (\frac{1}{e^{(n)}_1(m-1) } + 1 + e^{(n)}_0(m-1) \right ) 
		\prod_{\ell=1}^{n+1}  \tau^{(n)}_0(\ell), 
\end{gather*}
we have an estimate 
\begin{gather*}
	|\chi^{(n)}_0(m-1)| 
		< \left ( \frac{1}{e^{(n)}_1(0) } + 1  + e^{(n)}_0(0) \right ) \left ( \frac{3}{4} \right )^{3n + 4} ,
\end{gather*}
where we used Lemma \ref{lem:e-decrease}. 
This implies that $\tau^{(n)}_0(m)\chi^{(n)}_0(m-1)$ converges uniformly to $0$ as $n \to \infty$. 
By the same argument as in the proof of Lemma \ref{lem:gamma}(1), 
we check that the sequence $\left\{ \lambda^{(n)}_0(m) \right\}_{n, m}$ is bounded. 
Then, by using the recurrence relation (\ref{eq:lambda0-recurrence}) 
and the same argument as in the proof of Lemma \ref{lem:gamma} (3), 
we know that there exists $\lambda_0(m)=\displaystyle \lim_{n \to \infty} \lambda^{(n)}_0(m)$ 
and this convergence is uniform with respect to $m$. 
Furthermore, $\nu^{(n)}_0$ converges to $0$ as $n \to \infty$ 
because of $|\nu^{(n)}_0| < (3/4)^{3n+4}$. 
By these facts, 
both $\varphi^{(n)}_0(m)$ and $\Phi^{(n)}_0(m)$ converge uniformly to 
\begin{gather*}
	\varphi_0(m) = 1 + \frac{1 + e_0(m)}{\lambda_0(m) + 1 + e_0(m)} 
	\qquad \text{and} \qquad 
	\Phi_0(m) = 1 + \frac{\lambda_0(m) + e_0(m)}{1 + e_0(m)} , 
\end{gather*}
respectively, as $n \to \infty$. 
The cases $j = 1$ and $j = 2$ are similarly treated. 
\end{proof}

Let 
\begin{gather*}
	\sigma_j(m) = \lim_{n \to \infty} \sigma^{(n)}_j(m), \qquad 
	\tau_j(m) = \lim_{n \to \infty} \tau^{(n)}_j(m) , \qquad 
	\lambda_j(m)  = \lim_{n \to \infty} \lambda_j(m)
\end{gather*}
for $j = 1,2$. By (\ref{eq:lambda0-recurrence}) and (\ref{eq:lambda12-recurrence}), the following recurrence relation holds for all $j = 0, 1,2$:
\begin{gather} \label{eq:lambda-recurrence}
	\lambda_j(m) = \sigma_j(m) + \tau_j(m)\lambda_j(m-1) . 
\end{gather}

\begin{lem} \label{lem:0case}
	For $m \geq 0$, we have 
	\begin{gather*}
		\dfrac{331}{200} < \varphi_0(m +1) , \qquad \dfrac{5}{3} < \Phi_0(m) . 
	\end{gather*}
\end{lem}

\begin{proof}
Since
\begin{gather*}
	\dfrac{331}{200}<1+\dfrac{1+e_0(m)}{\lambda_0(m)+1+e_0(m)}
	\iff \lambda_0(m) < \dfrac{69}{131}(1+e_0(m)) ,
\end{gather*}
we prove the right hand side inequality by induction. 
For $m = 1$, we have
\begin{align*}
	& \lambda_0(1)=\dfrac{19+\sqrt{10}}{27}=0.82082\cdots , \\
	& \dfrac{69}{131}(1+e_0(1))=\dfrac{23(188-\sqrt{10})}{5109}=0.83211\cdots . 
\end{align*}
If we assume $\lambda_0(m-1) < 69(1+e_0(m-1))/131$, then by (\ref{eq:lambda-recurrence})
\begin{align*}
	\lambda_0(m) &= \sigma_0(m)+\tau_0(m) \lambda_0(m-1) \\
		&< \sigma_0(m)+\tau_0(m) \cdot \dfrac{69}{131}(1+e_0(m-1)) .
\end{align*}
To complete this induction, it is sufficient to prove
\begin{gather*}
	\sigma_0(m)+\tau_0(m) \cdot \dfrac{69}{131}(1+e_0(m-1)) < \dfrac{69}{131}(1+e_0(m)) . 
\end{gather*}
It is easy to check that the truth of this inequality is equivalent to the truth of 
\begin{gather*}
 100e_0(m)+100(2e_0(m)-1)e_0(m-1) < 69 . 
\end{gather*}
By Lemma \ref{lem:e-decrease}, $e_0(1) \leq e_0(m) \leq e_0(0)$ holds for all $m \geq 1$, and this gives  
\begin{align*}
	& 100e_0(m)+100(2e_0(m)-1)e_0(m-1) \\
		&\leq 100e_0(0)+200(e_0(0))^2-100e_0(1) \\
			&=\dfrac{100(1545789835+22158874 \sqrt{10})}{2387184813} 
		=67.68\cdots \\
		&<69 . 
\end{align*}

Next we prove the second inequality. 
Since
\begin{gather*}
	\dfrac{5}{3}<1+\dfrac{\lambda_0(m)+e_0(m)}{1+e_0(m)}
	\iff \dfrac{1}{3}(2-e_0(m))<\lambda_0(m) , 
\end{gather*}
we prove the inequality of the right hand side by induction on $m$. 
When $m = 0$, this is satisfied since 
\begin{align*}
	& \lambda_0(0)=\dfrac{10+\sqrt{10}}{15}=0.877485\cdots , \\
	& \dfrac{1}{3}(2-e_0(0))=\dfrac{1}{3} \left( 2-\dfrac{\sqrt{2}}{\sqrt{5}} \right)=0.455848\cdots . 
\end{align*}
If we assume $(2-e_0(m-1))/3<\lambda_0(m-1)$, then  by (\ref{eq:lambda-recurrence})
\begin{align*}
	\lambda_0(m) &= \sigma_0(m)+\tau_0(m) \lambda_0(m-1) \\
		&> \sigma_0(m)+\tau_0(m) \cdot \dfrac{1}{3}(2-e_0(m-1)) . 
\end{align*}
Therefore, it is sufficient to check 
\begin{gather*}
	\sigma_0(m)+\tau_0(m) \cdot \dfrac{1}{3}(2-e_0(m-1)) > \dfrac{1}{3}(2-e_0(m))  
\end{gather*}
for $m \geq 1$.  The truth of this inequality is equivalent to the truth of 
\begin{gather*}
 4e_0(m)+(2e_0(m)-1)e_0(m-1) > 2 . 
\end{gather*}
By Lemma \ref{lem:e-decrease}, $e_0(1) \leq e_0(m) \leq e_0(0)$ holds for all $m$, 
thus we obtain
\begin{align*}
	4e_0(m)+(2e_0(m)-1)e_0(m-1) &\geq 4e_0(1)+2(e_0(1))^2-e_0(0) \\
		&=\dfrac{216650-17449\sqrt{10}}{68445} 
		=2.3591\cdots \\
		&>2 . 
\end{align*}
\end{proof}

\begin{lem} \label{lem:12case}
	For $m \geq 0$ and $j = 1,2$, we have 
	\begin{gather*}
		\dfrac{5}{3} < \varphi_j(m) , \qquad \dfrac{5}{3} < \Phi_j(m). 
	\end{gather*}
\end{lem}

\begin{proof}
These inequalities are proved by the same argument as in the proof of Lemma \ref{lem:0case}. 
We omit the details. 
\end{proof}

\begin{lem} \label{lem:asm(2)}
	There exists $\delta>0$ such that inequalities
		\begin{alignat*}{2}
			\varphi_0(0) + \delta & < \varphi_0(m+1), & \qquad \varphi_0(0)+\delta & \leq \Phi_0(m) \\  
			\varphi_0(0)+\delta & \leq \varphi_j(m),& \qquad \varphi_0(0)+\delta & \leq \Phi_j(m) . 
	\end{alignat*}
\end{lem}
hold for all $m \geq 0$ and  $j=1, 2$ .	

\begin{proof} 
Since $\varphi_0(0) = \lim_{n \to \infty} \zeta^{(n)}_0 = r_1$, the value
\begin{gather*}
	\delta=\dfrac{331}{200}-\varphi_0(0)
		=\dfrac{331}{200}-\dfrac{48+\sqrt{10}}{31}
		>0 
\end{gather*}
satisfies the desired inequalities by Lemmas \ref{lem:0case} and \ref{lem:12case}. 
\end{proof}

By Lemmas \ref{lem:asm(1)} and \ref{lem:asm(2)}, 
the double series  $\left \{ \varphi^{(n)}_i(m) \right \}_{n,m}$ and $\left \{ \Phi^{(n)}_i(m) \right \}_{n,m}$ 
for $i = 0, 1, 2$ satisfy the  assumptions of Lemma \ref{lem:memo2}, 
whence 
\begin{gather*}
	 \varphi^{(n)}_0(0) = \displaystyle \min_{0 \leq m \leq n} \left \{ \varphi^{(n)}_0(0), \Phi^{(n)}_0(0), \varphi^{(n)}_0(m+1), 
			\Phi^{(n)}_0(m+1), \varphi^{(n)}_1(m), \Phi^{(n)}_1(m), \varphi^{(n)}_2(m), \Phi^{(n)}_2(m) \right \}
\end{gather*}
for all sufficiently large $n$. Step 3 is completed and we finish the proof of Theorem \ref{th:lim-xn}.

\vskip 10mm
\section{The largest accumulation point of $\rep(\mathbf{St})$} \label{sec:largest}

In this section, we prove

\begin{thm}\label{th:accumulation}
	$r_1$ is an accumulation point of $\rep(\mathbf{St})$. 
\end{thm}

We start with another description of $\zeta^{(n)}_0 = \varphi^{(n)}_0(0)$. 
The group $\mathrm{GL}_2(\mathbb{Q})$ acts 
on the set $\mathbb{R} \setminus \mathbb{Q}$ of irrational numbers 
by linear fractional transformations, 
i.e., 
\begin{gather*}
	(a_{ij})* \xi = \frac{a_{11}\xi + a_{12}} {a_{21}\xi + a_{22}}, \qquad 
	(a_{ij})  \in \mathrm{GL}_2(\mathbb{Q}), \;\; \xi \in \mathbb{R} \setminus \mathbb{Q}. 
\end{gather*}
Let 
\begin{gather*}
	A = \begin{pmatrix}
			1 & 3 \\
			2 & 5
		\end{pmatrix}
	\qquad \text{and} \qquad
	A^m  = \begin{pmatrix}
			a_m & b_m \\
			c_m & d_m
		\end{pmatrix}
\end{gather*}
for $m \ge 1$. 
Eigenvalues of $A$ are $p = 3 - \sqrt{10}$ and $q = 3 + \sqrt{10}$, 
and matrix calculation gives
\begin{gather*}
	A^m = \frac{1}{2\sqrt{10}} 
		\begin{pmatrix}
			(\sqrt{10}+2)p^m + (\sqrt{10} - 2)q^m & 3(q^m - p^m) \\ 
			2(q^m-p^m) & (\sqrt{10} - 2)p^m + (\sqrt{10} + 2)q^m
		\end{pmatrix} = \frac{q^m}{2\sqrt{10}}A_0(m) .
\end{gather*}
Let $z$ be a variable.  
We define $z_m$ and $y_m$ for $m = 0, 1, 2, \cdots$ as follows: 
\begin{gather*}
	z_0 = z, \qquad 
	z_m = A^m*z = \frac{a_m z + b_m}{c_m z + d_m}, \qquad 
	y_m  = \frac{z_m(2z_{m+1} -1)}{z_{m+1}}.
\end{gather*}
Since $z_m$ is satisfying $ 2z_{m+1}z_{m} = 3 + z_{m} - 5z_{m+1}$, 
$y_m$ is represented as 
\begin{gather*}
	y_m = \frac{3}{z_{m+1}} - 5 
		= \begin{pmatrix}
			-5 & 3 \\ 
			1 & 0 
		\end{pmatrix} 
		* z_{m+1} . 
\end{gather*}
For each positive integer $n$, let
\begin{gather*}
	f_n(z) = 2 + 2 \sum_{j=0}^{n-1} \prod_{i = j+1}^n y_{i} + \prod_{i = 0}^n y_{i} ,  \qquad 
	g_n(z) = z \prod_{j=1}^{n+1} (2z_{j} - 1) , \\
	h_n(z) = \frac{1}{1+z} \cdot \frac{z \cdot z_{n+1} \cdot f_n(z)}{1 - g_n(z)} .
\end{gather*}

\begin{lem}
$\varphi^{(n)}_0(0)$ is represented as 
\begin{gather*}
	\varphi^{(n)}_0(0) = 1 + \frac{1}{h_n(e^{(n)}_0(0)) + 1}. 
\end{gather*}
\end{lem}

\begin{proof}
From (\ref{eq:e-gen}),   it follows 
\begin{gather*}
e^{(n)}_0(m) = A^m*e^{(n)}_0(0)  = z_m \big\vert_{z = e^{(n)}_0(0)}.
\end{gather*}
Since 
\begin{gather*}
e^{(n)}_1(m) = \frac{1}{2 + e^{(n)}_0(m)}, \qquad e^{(n)}_2(m) = \frac{1}{e^{(n)}_0(m+1)}  -  1,
\end{gather*}
we have 
\begin{gather*}
\tau^{(n)}_0(m) = e^{(n)}_0(m)e^{(n)}_2(m-1)e^{(n)}_1(m-1) = 2e^{(n)}_0(m) - 1 . 
\end{gather*}
This implies
\begin{gather*}
 \nu^{(n)}_0 = e^{(n)}_0(0) \prod_{j=1}^{n+1} \tau^{(n)}_0(j) = g_n(e^{(n)}_0(0)).
\end{gather*}
By the definition (\ref{eq:gamma}) of $\gamma^{(n)}(m)$, we have
\begin{align*}
 \gamma^{(n)}(n) & = 2 + \tau^{(n)}(n)\gamma^{(n)}(n-1) = 2 + \tau^{(n)}(n) \left ( 2 + \tau^{(n)}(n-1)\gamma^{(n)}(n-2) \right )  \\
 & = 2 + 2\sum_{ j=0}^{n-1} \prod_{i = j+1}^n \tau^{(n)}(i) + \prod_{i=0}^n \tau^{(n)}(i). 
\end{align*}
Since 
\begin{gather*}
\tau^{(n)}(m) = e^{(n)}_0(m)e^{(n)}_1(m)e^{(n)}_2(m) = \frac{e^{(n)}_0(m)}{e^{(n)}_0(m+1)}\tau^{(n)}_0(m+1) = \frac{z_m}{z_{m+1}}(2z_{m+1} -1) \big\vert_{z = e^{(n)}_0(0)}, 
\end{gather*}
we obtain
\begin{gather*}
\gamma^{(n)}(n) = f_n(e^{(n)}_0(0)),  
\end{gather*}
and hence
\begin{gather*}
\lambda^{(n)}_0(0) = e^{(n)}_0(0)e^{(n)}_0(n+1)\gamma^{(n)}(n) = z \cdot z_{n+1} \cdot f_n(z) \big\vert_{z = e^{(n)}_0(0)}.
\end{gather*}
Therefore, 
\begin{gather*}
\varphi^{(n)}_0(0)  = 1 + \dfrac{1 + e^{(n)}_0(0)}{\dfrac{\lambda^{(n)}_0(0)}{1 - \nu^{(n)}_0} + 1 + e^{(n)}_0(0)}  =  1 + \frac{1}{h_n(e^{(n)}_0(0)) + 1}. 
\end{gather*}
\end{proof}
 
For every $n$, we put 
\begin{gather*}
	a_{n+1}' = 2\sum_{j=1}^{n+1} a_j + 1, \qquad 
	b_{n+1}' = 2 \sum_{j=1}^{n+1} b_j .
\end{gather*}

\begin{lem} \label{lem:1}
	Functions $f_n(z), g_n(z)$ and $h_n(z)$ are rational functions of the following forms: 
	\begin{gather*}
		f_n(z) = \frac{a_{n+1}' z +b_{n+1}'} {a_{n+1} z + b_{n+1}} , \qquad 
		g_n(z) = \frac{z}{c_{n+1}z + d_{n+1}}, \qquad 
		h_n(z) = \frac{z}{1+z} \cdot \frac{a_{n+1}' z + b_{n+1}'}{(c_{n+1}-1)z + d_n} .
	\end{gather*}
\end{lem}

\begin{proof}
From
\begin{gather*}
	\begin{pmatrix}
		-5 & 3 \\
		1 & 0
	\end{pmatrix}
	A^{m+1} =  
	\begin{pmatrix}
		-5 & 3 \\
		1 & 0
	\end{pmatrix} 
	A
	\begin{pmatrix}
		a_m & b_m \\
		c_m & d_m
	\end{pmatrix}
	= \begin{pmatrix}
		a_m & b_m \\
		a_{m+1} & b_{m+1}
	\end{pmatrix} ,
\end{gather*}
it follows 
\begin{gather*}
	y_{m} = \begin{pmatrix}
			-5 & 3 \\
			1 & 0
		\end{pmatrix}
		A^{m+1} *z
	= \frac{a_m z + b_m}{a_{m+1}z + b_{m+1}},
\end{gather*}
and hence
\begin{gather*}
	y_{n} \cdots y_{i} = \frac{a_i z + b_i}{a_{n+1}z + b_{n+1}}. 
\end{gather*}
This gives
\begin{gather*}
	f_n(z) = 2 + 2 \sum_{j=0}^{n-1} \frac{a_{j+1} z + b_{j+1}}{a_{n+1}z + b_{n+1}} + \frac{z}{a_{n+1}z + b_{n+1}} = \frac{a_{n+1}' z +b_{n+1}'} {a_{n+1} z + b_{n+1}}. 
\end{gather*}
Similarly, from 
\begin{gather*}
	2z_{m} - 1 = \begin{pmatrix}
				2 & -1 \\
				0 & 1
			\end{pmatrix}
			A^{m}*z
		= \frac{(2a_m - c_m) z + (2b_m - d_m)}{c_m z + d_m}, 
\end{gather*}
it follows
\begin{gather*}
	2z_{m+1} - 1 = \begin{pmatrix}
				2 & -1 \\
				0 & 1
			\end{pmatrix}
			A A^{m}*z
		= \frac{c_mz + d_m} {c_{m+1}z + d_{m+1}} .
\end{gather*}
This gives
\begin{gather*}
	g_n(z) = z \cdot \frac{(2a_1 - c_1)z + (2b_1 - d_1)}{c_{n+1}z + d_{n+1}}
		= \frac{z}{c_{n+1}z + d_{n+1}} .
\end{gather*}
The equality of $h_n(z)$ is also easy. 
\end{proof}

Next we show that 
the quadratic irrational $e^{(n)}_0(0) = [0,\overline{1,(1,1,2)^{n+1}}]$ is not contained 
in the quadratic number field $\mathbb{Q}(\sqrt{10})$ when $n$ is odd or $n \not\equiv \pm 1 \mod 5$. 
Let $\theta_n = [\overline{1,(1,1,2)^{n+1}}] =1/  e^{(n)}_0(0) $ and 
$\theta = [1,\overline{1,1,2}] = \sqrt{10}/{2} $. 
The $j$-th convergent of $\theta_n$ is equal to the $j$-the convergent  $p_j/q_j $  of $\theta$ if $j \leq 3n+4$. 
Since $\theta_n = [1,(1,1,2)^{n+1}, \theta_n]$, we obtain
\begin{gather*}
	\theta_n = \frac{p_{3n+3} \theta_n + p_{3n+2}}{q_{3n+3}\theta_n + q_{3n+2}} ,
\end{gather*}
and hence
\begin{gather*}
	\theta_n = \frac{p_{3n+3} - q_{3n+2} + \sqrt{D_n}}{2q_{3n+3}} , 
\end{gather*}
where 
\begin{gather*}
	D_n = (q_{3n+2}-p_{3n+3})^2 + 4q_{3n+3}p_{3n+2}
		= (q_{3n+2} + p_{3n+3})^2 - 4(-1)^{3n+4} .
\end{gather*}
Let
\begin{gather*}
	B_1 = \begin{pmatrix}
			1 & 1 \\
			1 & 0
		\end{pmatrix}, \qquad
	B_2 =  \begin{pmatrix}
			2 & 1 \\
			1 & 0
		\end{pmatrix}, \qquad
	B = B_1^2 B_2 = \begin{pmatrix}
			5 & 2 \\
			3 & 1
		\end{pmatrix} .
\end{gather*}
Then we have 
\begin{gather*}
	B_1 B^{n+1} = \begin{pmatrix}
			p_{3n+3} & p_{3n+2} \\
			q_{3n+3} & q_{3n+2}
		\end{pmatrix}
\end{gather*}
and  
\begin{gather*}
	D_n = \tr (B_1B^{n+1})^2 - 4(-1)^n . 
\end{gather*}

\begin{lem}\label{lem:discriminant}
	For each $n$, the following congruent conditions hold: 
	\begin{gather*}
		D_n \equiv \begin{cases}
				(-1)^k & (n = 5k) \\
				2(-1)^k & (n= 5k+2) \\
				-2(-1)^k & (n = 5k+3) \\
				0  & (n = 5k+1, 5k+4)
			\end{cases}
			\mod 5 \qquad \text{and} \qquad
		D_n \equiv \begin{cases}
				1 & (\text{$n$ odd}) \\
				0 & (\text{$n$ even})
			\end{cases}
			\mod 2. 
	\end{gather*}
\end{lem}

\begin{proof}
Let $\ol B  = B \mod 5$. 
Then we have 
\begin{gather*}
	\ol B = \begin{pmatrix}
				\ol 0 & \ol 2 \\
				\ol 3 & \ol 1
			\end{pmatrix}, \quad
	\ol B^2 = \begin{pmatrix}
				\ol 1 & \ol 2 \\
				\ol 3 & \ol 2
			\end{pmatrix} , \quad
	\ol B^3 = \begin{pmatrix}
				\ol 1 & \ol 4 \\
				\ol 1 & \ol 3
			\end{pmatrix} , \quad 
	\ol B^4 = \begin{pmatrix}
				\ol 2 & \ol 1 \\
				\ol 4 & \ol 0
			\end{pmatrix}, \quad
	\ol B^5 = \begin{pmatrix}
				\ol 3 & \ol 0 \\
				\ol 0 & \ol 3
			\end{pmatrix}. 
\end{gather*}
Therefore, 
\begin{gather*}
	\tr(B_1B^{5k+j}) \mod 5 = \begin{cases}
			\ol 0 & (j=1) \\
			\ol 3^k & (j=0,2,3) \\
			\ol 2 \cdot \ol 3^k & (j = 4)
		\end{cases}. 
\end{gather*}
This gives the congruent condition of $D_n$ modulo $5$. 
The case of modulo $2$ is treated similarly. 
\end{proof}

\begin{lem} \label{lem:3}
	If $n$ is odd or $n \not\equiv \pm 1 \mod 5$, 
	then $e^{(n)}_0(0)$ is not contained in $\mathbb{Q}(\sqrt{10})$. 
\end{lem}

\begin{proof}
If we suppose $e^{(n)}_0(0) \in \mathbb{Q}(\sqrt{10})$, 
then $\sqrt{D_n}$ must be an element of $\mathbb{Q}(\sqrt{10})$. 
Let $\sqrt{D_n} = a + b\sqrt{10}$ for $a, b \in \mathbb{Q}$. 
From $D_n = a^2 + 10b^2 + 2ab\sqrt{10}$ and the irrationality of $\sqrt{D_n}$, 
it follows  $a = 0$. 
Then $D_n = 10b^2$. 
This implies $b \in \mathbb{Z}$ and $D_n \equiv 0 \mod 10$. 
This contradicts to Lemma \ref{lem:discriminant}. 
\end{proof}

\begin{prop}\label{prop:less}
	Let $n$ be a sufficiently large positive integer and 
	assume that $n$ is odd or $n \not\equiv \pm 1 \mod 5$. 
	Let $x^{(n)}$ be a Sturmian word of slope $[0, a_1, \cdots , a_K, \overline{(2,1,1)^{n+1},1}]$. 
	Then  $\rep(x^{(n)})$ is strictly less than $r_1$. 
\end{prop}

\begin{proof}
By Theorem \ref{th:repmax2}, $\rep(x^{(n)})$ satisfies 
\begin{gather}
	\rep(x^{(n)}) \leq r_1 = \frac{48 + \sqrt{10}}{31} . \label{eq:r2}
\end{gather}
Suppose that the equality holds in (\ref{eq:r2}). 
By Lemmas \ref{buglem:7.2} and \ref{buglem:7.4}, there exists $K_0 \ge K$ such that the case $[2]_k$
holds for all  $k > K_0$. 
Then, by Step 3 of the proof of Theorem \ref{th:lim-xn}, we have 
\begin{gather*}
	r_1 = \rep(x^{(n)}) = \varphi^{(n)}_0(0)
		=  1 + \frac{1}{h_n(e^{(n)}_0(0)) + 1}
\end{gather*}
if  $n$ is sufficiently large. 
Then $h_n(e^{(n)}_0(0)) $ must be contained in $\mathbb{Q}(\sqrt{10})$. 
By Lemma \ref{lem:1} and the fact that $e^{(n)}_0(0)$ is a quadratic irrational, 
$ h_n(e^{(n)}_0(0)) $ is represented as 
\begin{gather*}
	h_n(e^{(n)}_0(0))  = \frac{e^{(n)}_0(0)}{1+e^{(n)}_0(0)} \cdot 
			\frac{a_{n+1}' e^{(n)}_0(0) + b_{n+1}'}{(c_{n+1}-1)e^{(n)}_0(0) + d_n} 
		=  \frac{a'' e^{(n)}_0(0) +b''}{c'' e^{(n)}_0(0) + d''}, 
\end{gather*}
where $a'', b'', c'', d'' \in \mathbb{Q}$. 
This implies $e^{(n)}_0(0) \in \mathbb{Q}(\sqrt{10})$ and a contradiction. 
\end{proof}

\begin{proof}[Proof of Theorem \ref{th:accumulation}.]
Bugeaud and Kim constructed an example of a Sturmian word $x$ with $\rep(x) = r_{\max}$. 
We use the same method as in \cite[\S 7]{bugeaud}. 
For every odd number $n \geq 1$, set $\vartheta_n = [0,a_1^{(n)},a_2^{(n)}, \cdots] = [0,\overline{(2,1,1)^{n+1},1}]$. 
Let $\{M_k^{(n)}\}_k$ be the characteristic block defined from $\{a_k^{(n)} \}_k$. 
For  $k \geq 2$, the finite word $W^{(n)}_k = 1M^{(n)}_0M^{(n)}_1 \cdots M^{(n)}_{k-2}$ is a suffix of $M^{(n)}_k$. 
Define 
\begin{gather*}
	s^{(n)} = \lim_{k \to \infty} W^{(n)}_k = \lim_{k \to \infty}(1M^{(n)}_0M^{(n)}_1 \cdots M^{(n)}_{k-2}) . 
\end{gather*}
The slope of $s^{(n)}$ equals $\vartheta_n$ and 
$s^{(n)} = W^{(n)}_kM^{(n)}_{k-1}M^{(n)}_k\widetilde M^{(n)}_k \cdots$ holds for all  $k \geq 2$. 
By Proposition \ref{prop:less}, $\rep(s^{(n)})$ is strictly less than $r_1$ for sufficiently large $n$. 
By Theorem \ref{th:lim-xn}, $\rep(s^{(n)})$ converges to $r_1$ as $n \to \infty$. 
\end{proof}

\end{document}